\documentclass[11pt]{article}
\usepackage{amsfonts,mathrsfs,amssymb,amsthm,mathptm}
\usepackage{amsmath,amscd}
\usepackage{mathptm,pslatex}
\usepackage{eucal}
\usepackage{color}
\usepackage[all]{xy}
\usepackage{graphicx}
\usepackage{fancyhdr}
\usepackage{url}
\usepackage{hyperref}
\hypersetup{
 colorlinks=true,
 linkcolor=black,
 citecolor=black,
 urlcolor=black
}

\begin{document}
\newtheorem{Def}{Definition}[section]
\newtheorem{Bsp}[Def]{Example}
\newtheorem{Prop}[Def]{Proposition}
\newtheorem{Theo}[Def]{Theorem}
\newtheorem{Lem}[Def]{Lemma}
\newtheorem{Koro}[Def]{Corollary}
\newtheorem{Nota}[Def]{Notation}
\theoremstyle{definition}
\newtheorem{Rem}[Def]{Remark}
\newtheorem{Prob}[Def]{Problem}
\newtheorem{Ques}[Def]{Question}

\newcommand{\add}{{\rm add}}
\newcommand{\con}{{\rm con}}
\newcommand{\gd}{{\rm gldim}}
\newcommand{\repdim}{{\rm repdim}}
\newcommand{\sd}{{\rm stdim}}
\newcommand{\sr}{{\rm sr}}
\newcommand{\dm}{{\rm domdim}}
\newcommand{\cdm}{{\rm codomdim}}
\newcommand{\tdim}{{\rm dim}}
\newcommand{\E}{{\rm E}}
\newcommand{\K}{{\rm k}}
\newcommand{\Mor}{{\rm Morph}}
\newcommand{\End}{{\rm End}}
\newcommand{\ind}{{\rm ind}}
\newcommand{\rsd}{{\rm res.dim}}
\newcommand{\rd} {{\rm rd}}
\newcommand{\ol}{\overline}
\newcommand{\overpr}{$\hfill\square$}
\newcommand{\rad}{{\rm rad}}
\newcommand{\soc}{{\rm soc}}
\renewcommand{\top}{{\rm top}}
\newcommand{\pd}{{\rm pdim}}
\newcommand{\id}{{\rm idim}}
\newcommand{\fld}{{\rm fdim}}
\newcommand{\Fac}{{\rm Fac}}
\newcommand{\Gen}{{\rm Gen}}
\newcommand{\fd} {{\rm fin.dim}}
\newcommand{\Fd} {{\rm Fin.dim}}
\newcommand{\Pf}[1]{{\mathcal P}^{<\infty}(#1)}
\newcommand{\DTr}{{\rm DTr}}
\newcommand{\cpx}[1]{#1^{\bullet}}
\newcommand{\D}[1]{{\mathcal D}(#1)}
\newcommand{\Dz}[1]{{\mathcal D}^+(#1)}
\newcommand{\Df}[1]{{\mathcal D}^-(#1)}
\newcommand{\Db}[1]{{\mathscr D}^b(#1)}
\newcommand{\C}[1]{{\mathcal C}(#1)}
\newcommand{\Cz}[1]{{\mathcal C}^+(#1)}
\newcommand{\Cf}[1]{{\mathcal C}^-(#1)}
\newcommand{\Cb}[1]{{\mathcal C}^b(#1)}
\newcommand{\Dc}[1]{{\mathcal D}^c(#1)}
\newcommand{\Ds}[1]{{\mathscr D}_{sg}(#1)}
\newcommand{\Kz}[1]{{\mathcal K}^+(#1)}
\newcommand{\Kf}[1]{{\mathcal  K}^-(#1)}
\newcommand{\Kb}[1]{{\mathscr K}^b(#1)}
\newcommand{\DF}[1]{{\mathcal D}_F(#1)}
\newcommand{\Gp}[1]{{\mathscr G}(#1)}
\newcommand{\rpd}{{\rm repdim}}

\newcommand{\Kac}[1]{{\mathcal K}_{\rm ac}(#1)}
\newcommand{\Keac}[1]{{\mathcal K}_{\mbox{\rm e-ac}}(#1)}

\newcommand{\modcat}{\ensuremath{\mbox{{\rm -mod}}}}
\newcommand{\Modcat}{\ensuremath{\mbox{{\rm -Mod}}}}
\newcommand{\Spec}{{\rm Spec}}

\newcommand{\stmc}[1]{#1\mbox{{\rm -{\underline{mod}}}}}
\newcommand{\Stmc}[1]{#1\mbox{{\rm -{\underline{Mod}}}}}
\newcommand{\prj}[1]{#1\mbox{{\rm -proj}}}
\newcommand{\inj}[1]{#1\mbox{{\rm -inj}}}
\newcommand{\Prj}[1]{#1\mbox{{\rm -Proj}}}
\newcommand{\Inj}[1]{#1\mbox{{\rm -Inj}}}
\newcommand{\PI}[1]{#1\mbox{{\rm -Prinj}}}
\newcommand{\GP}[1]{#1\mbox{{\rm -GProj}}}
\newcommand{\GI}[1]{#1\mbox{{\rm -GInj}}}
\newcommand{\gp}[1]{#1\mbox{{\rm -Gproj}}}
\newcommand{\gi}[1]{#1\mbox{{\rm -Ginj}}}
\newcommand{\cre}{{\rm CR}}
\newcommand{\Stgp}[1]{#1\mbox{{\rm -{\underline{Gproj}}}}}

\newcommand{\opp}{^{\rm op}}
\newcommand{\otimesL}{\otimes^{\rm\mathbb L}}
\newcommand{\rHom}{{\rm\mathbb R}{\rm Hom}\,}
\newcommand{\pdim}{\pd}
\newcommand{\Hom}{{\rm Hom}}
\newcommand{\Coker}{{\rm Coker}}
\newcommand{ \Ker  }{{\rm Ker}}
\newcommand{ \Cone }{{\rm Con}}
\newcommand{ \Img  }{{\rm Im}}
\newcommand{\Ext}{{\rm Ext}}
\newcommand{\StHom}{{\rm \underline{Hom}}}
\newcommand{\StEnd}{{\rm \underline{End}}}

\newcommand{\KK}{I\!\!K}

\newcommand{\gm}{{\rm _{\Gamma_M}}}
\newcommand{\gmr}{{\rm _{\Gamma_M^R}}}

\def\vez{\varepsilon}\def\bz{\bigoplus}  \def\sz {\oplus}
\def\epa{\xrightarrow} \def\inja{\hookrightarrow}

\newcommand{\lra}{\longrightarrow}
\newcommand{\llra}{\longleftarrow}
\newcommand{\lraf}[1]{\stackrel{#1}{\lra}}
\newcommand{\llaf}[1]{\stackrel{#1}{\llra}}
\newcommand{\ra}{\rightarrow}
\newcommand{\Ra}{\Rightarrow}
\newcommand{\La}{\Leftarrow}
\newcommand{\xr}{\xrightarrow}
\newcommand{\dk}{{\rm dim_{_{k}}}}

\newcommand{\holim}{{\rm Holim}}
\newcommand{\hocolim}{{\rm Hocolim}}
\newcommand{\colim}{{\rm colim\, }}
\newcommand{\limt}{{\rm lim\, }}
\newcommand{\Add}{{\rm Add }}
\newcommand{\Prod}{{\rm Prod }}
\newcommand{\Tor}{{\rm Tor}}
\newcommand{\Cogen}{{\rm Cogen}}
\newcommand{\Tria}{{\rm Tria}}
\newcommand{\Loc}{{\rm Loc}}
\newcommand{\Coloc}{{\rm Coloc}}
\newcommand{\tria}{{\rm tria}}
\newcommand{\Con}{{\rm Con}}
\newcommand{\Thick}{{\rm Thick}}
\newcommand{\thick}{{\rm thick}}
\newcommand{\Sum}{{\rm Sum}}
\newcommand{\N}{\mathbb{N}}
\newcommand{\B}{\mathcal{B}}
\newcommand{\A}{\mathcal{A}}
\newcommand{\copres}{\ensuremath{\mbox{{\rm copres}}}}

\centerline{\textbf{Singular equivalences and homological conjectures}}

\medskip
\centerline{Zhenxian Chen and Changchang Xi$^{*}$}

\begin{center}School of Mathematical Sciences, Capital Normal University \\ 100048  Beijing, P. R. China\end{center}

\begin{center}\small{Email: czx18366459216@163.com (Z.X.Chen), xicc@cnu.edu.cn (C.C.Xi)}\end{center}
\renewcommand{\thefootnote}{\alph{footnote}}
\setcounter{footnote}{-1} \footnote{ $^*$ Corresponding author.
Email: xicc@cnu.edu.cn; Fax: 0086 10 68903637.}
\renewcommand{\thefootnote}{\alph{footnote}}
\setcounter{footnote}{-1}
\renewcommand{\thefootnote}{\alph{footnote}}
\setcounter{footnote}{-1}
\begin{abstract}
The fact that each finite-dimensional algebra over a field is isomorphic to the centralizer of two matrices, has suggested to investigate representation theoretical problems of finite-dimensional algebras through the centralizer algebras of matrices. Therefore the first natural question is to study the problems for the centralizer algebra of one matrix, called a centralizer matrix algebra. In this paper we give an elementary and explicit approach to the singularity categories and singular equivalences of centralizer matrix algebras, and verify the Auslander--Reiten and Cartan determinant conjectures for centralizer matrix algebras. Consequently, all historical homological conjectures (the finitistic dimension, Wakamatsu tilting, tilting (projective) complement, strong Nakayama, generalized Nakayama and Nakayama conjectures) are true for centralizer matrix algebras over fields. Moreover, we prove some homological invariants of singular equivalences of centralizer matrix algebras.
 \end{abstract}

{\footnotesize\tableofcontents\label{contents}}

\section{Introduction}
The familiar fact that every finite-dimensional algebra over a field is isomorphic to the centralizer of two matrices in a full matrix algebra (see \cite{brenner1972}), has suggested to contemplate problems on representations of finite-dimensional algebras through the ones of centralizer algebras of matrices, instead of working with quivers and relations.
The first fundamental step toward this direction is to investigate the centralizer of one matrix. Such an algebra will be called a centralizer matrix algebra. In this case, a number of useful partial results have been obtained. The earliest result on this subject could trace back to Frobenius, who established an explicit dimension formula for a centralizer matrix algebra in terms of degrees of invariant factors of the given matrix \cite{frob1877} (see \cite[Theorems 1 and 2, p.105-106]{wed}). Recently, the structure of centralizer matrix algebras was discussed in \cite{xz-LAA}, it was shown that the centralizer matrix algebras of Jordan block matrices are cellular in the sense of Graham and Lehrer \cite{gl1996}. By a quite
different approach from \cite{xz-LAA, xz2}, complete characterizations of derived and stable equivalences were given in terms of  new types of matrix equivalence relations \cite{lx1, lx2}. Moreover, some major conjectures in representation theory, namely the finitistic dimension conjecture, Nakayama conjecture and Alperin--Auslander conjecture, are verified for centralizer matrix algebras \cite{lx1, xz3}.

In the present paper we have two aims in mind, the first one is to attack the problem of how to describe equivalences of singularity categories of centralizer matrix algebras over fields. Roughly speaking, the singularity category of an Artin algebra is the quotient of its
bounded by its perfect derived category. This algebraical construction goes back to Buchweitz (1987) and, independently, the geometrical formulation to Orlov (2004). The second aim, as a consequence of our discussions, is to show that the Cartan determinant and Auslander--Reiten conjectures hold for centralizer matrix algebras. Thus, together with results in \cite{lx1,lx2, xz2}, we know that all homological conjectures mentioned in \cite[Conjectures, p. 409]{ars} hold true for centralizer matrix algebras.

To state our results more precisely, we introduce now a few notations. Let $R$ be a field and $M_n(R)$ the full $n\times n$ matrix algebra over $R$. For a matrix $c\in M_n(R)$, we denote by $S_n(c,R)$ the \emph{centralizer matrix algebra}
$$ S_n(c,R):=\{x\in M_n(R)\mid cx=xc\}.$$

To describe singular equivalences of centralizer matrix algebras, we introduce a new type of equivalence relation on all square matrices over a field. This relation is called an $Sg$-equivalence,
defined in terms of irreducible factors of maximal elementary divisors of matrices together
with combinatorial data of multiplicities of elementary divisors. For more details, we refer to Section \ref{sect3}.

For an Artin algebra $A$, we denote by $\Ds{A}$ the singularity category of $A$, which is the Verdier quotient of the derived category $\Db{A}$ of $A$ by its perfect subcategory $\Kb{\prj{A}}$ (see Subsection \ref{sect2.1} for definition).

With these conventions established we now state our first main result.

\begin{Theo}{\rm [Theorem \ref{seocma}]}\label{thm1}
Let $R$ be a field and let $A:=S_n(c,R)$ and $B:=S_m(d,R)$ for $c\in M_n(R)$ and $d\in M_m(R)$. Then
the following are equivalent.

$\quad (i)$ $\Ds{A}$ and $\Ds{B}$ are equivalent as triangulated $R$-categories.

$\quad (ii)$ $\Ds{A}$ and $\Ds{B}$ are equivalent as $R$-categories.

$\quad (iii)$ $c$ and $d$ are $Sg$-equivalent as matrices.

\end{Theo}

The proof of this result is based on a complete description of the singularity category of $S_n(c,R)$ (see Theorem \ref{scocma}). For permutation matrices, any singular equivalence between their centralizer matrix algebras induces a singular equivalence of their singular parts. Now let us explain this point precisely.

Let $\Sigma_n$ denote the symmetric group of permutations on $\{1,2,\cdots,n\}$, and let $\sigma = \sigma_1 \cdots \sigma_s \in \Sigma_n$ be a product of disjoint cycles, and  $\lambda = (\lambda_1, \ldots, \lambda_s)$ be its cycle type. For a prime number $p>0$, the cycle $\sigma_i$ is said to be \emph{$p$-regular} if $p\nmid \lambda_i$, and \emph{$p$-singular} if $p\mid \lambda_i$. Let $r(\sigma)$ (respectively, $s(\sigma)$) be the product of the $p$-regular (respectively, $p$-singular) cycles of $\sigma$. We consider both $r(\sigma)$ and $s(\sigma)$ as elements in $\Sigma_n$.
Let $c_{\sigma}\in M_n(R)$ denote the permutation matrix of $\sigma\in \Sigma_n$.

For a positive integer $n$, let $\nu_p(n)$ be the maximal power index of $p$ such that $p^{\nu_p(n)}$ divides $n$.

\medskip
For the centralizer matrix algebras of permutation matrices, we have the following result.

\begin{Theo}{\rm [Theorem \ref{seispop}]}\label{thm2} Let $R$ be a field of characteristic $p>0$, and let $\sigma\in \Sigma_n$ and $\tau\in\Sigma_m$ be of cycle types $(\lambda_1, \ldots, \lambda_s)$ and $(\mu_1,\cdots,\mu_t)$, respectively.
Suppose $p\neq 2$ or $p=2$ and $\nu_p(\lambda_i)\neq 1\neq \nu_p(\mu_{j})$ for all $i\in [s]$ and $j\in[t]$. If $S_n(c_{\sigma},R)$ and $S_m(c_{\tau},R)$ are singularly equivalent, then so are $S_n(c_{s(\sigma)},R)$ and $S_m(c_{s(\tau)},R)$.
\end{Theo}

Furthermore, we give sufficient conditions in Corollary \ref{one-more} to ensure that Morita, derived, stable and singular equivalences induce each other between the centralizer matrix algebras of permutation matrices.

\medskip
In the representation theory and homological algebra of finite-dimensional algebras, homological conjectures have been a core set of problems.  Let us recall the not yet solved Auslander--Reiten conjecture or generalized Nakayama conjecture (ARC/GNC) and the Cartan determinant conjecture (CDC).

Let $\Lambda$ be an Artin algebra, and let $M$ be a finitely generated $\Lambda$-module. If $\Ext^i_{\Lambda}(M,M)=0$ for all $i>0$, then $M$ is said to be \emph{self-orthogonal}. If each indecomposable projective $\Lambda$-module is isomorphic to a direct summand of $M$, then $M$ is called a \emph{generator} over $\Lambda$.

\medskip
(ARC/GNC) If $M$ is a self-orthogonal generator over $\Lambda$, then $M$ is projective (see \cite{ar}).

(CDC) If the global dimension of $\Lambda$ is finite, then the Cartan determinant of $\Lambda$ is $1$ (see \cite{z1}).

\medskip
Though lots of efforts have been made in the last decades, these important conjectures are still open to date. Applying the ideas in this article, we will verify
these major conjectures for centralizer matrix algebras.

\begin{Theo}\label{hchfcma}
$(1)$ The Auslander--Reiten conjecture and the Cartan determinant conjecture hold true for centralizer matrix algebras over fields.

$(2)$ Let $R$ be a field, $c\in M_n(R)$ and $d\in M_m(R)$. If $S_n(c,R)$ and $S_m(d,R)$ are singularly equivalent, then

$\quad (i)$ $S_n(c,R)$ is quasi-hereditary if and only if $S_m(d,R)$ is quasi-hereditary.

$\quad (ii)$ The Cartan determinants of $S_n(c,R)$ and $S_m(d,R)$ are equal.
\end{Theo}

The contents of this article read as follows.  In Section \ref{sec2} we recall notions and terminologies, and prepare basic facts for proofs. In Section \ref{sect3} we introduce new types of equivalence relations on matrices over fields. We also compare them with some known equivalence relations. In Section \ref{sec4} we describe the singularity categories of centralizer matrix algebras. In Section \ref{sec5} we establish characterizations of singular equivalences of centralizer matrix algebras in terms of a new type of equivalence relations of matrices. This gives a proof of Theorem \ref{thm1}. In Section \ref{sec6} we investigate singular equivalences between centralizer matrix algebras of permutation matrices, and prove Theorem \ref{thm2}. In Section \ref{sect5.2} we prove Theorem \ref{hchfcma}.

\section{Preliminaries}\label{sec2}

In this section we recall some basic definitions and terminologies on exact categories, Frobenius categories and triangulated categories. For our later proofs, we prove some properties of modules over quotients of polynomial algebras.

Throughout this paper, $R$ denotes a field unless stated otherwise. By an algebra we mean a finite-dimensional unital associative algebra over $R$. All modules are finitely generated left modules.

\subsection{Exact and Frobenius categories}
Let $\mathcal C$ be an additive category. A full subcategory $\mathcal{B}$ of $\mathcal{C}$ is always assumed to be closed under isomorphisms, that is, if $X\in {\mathcal B}$ and $Y\in\cal C$ with $Y\simeq X$, then $Y\in{\mathcal B}$.

Let $X$ be an object in $\mathcal{C}$. We denote by $\add(X)$ the full subcategory of $\mathcal{C}$ consisting of all direct summands of the coproducts of finitely many copies of $X$.

The composition of morphisms $f: X\to Y$ and $g: Y\to Z$ in $\mathcal C$ is written as $fg: X\to Z$. The induced morphisms $\Hom_{\mathcal
C}(Z,f):\Hom_{\mathcal C}(Z,X)\ra \Hom_{\mathcal C}(Z,Y)$ and $\Hom_{\mathcal C}(f,Z): \Hom_{\mathcal C}(Y, Z)\ra \Hom_{\mathcal{C}}(X, Z)$ are denoted by $f^*$ and $f_*$, respectively, while the composition of functors $F:\mathcal{C}\to\mathcal{D}$ and $G:\mathcal{D}\to\mathcal{E}$ between categories is denoted  by $G\circ F$ which is a functor from $\mathcal C$ to $\mathcal E$.

Let $\mathcal{D}$ be a full subcategory of $\mathcal{C}$. A morphism $f: X\to Y$ in $\mathcal{C}$ is called a \emph{right minimal} morphism if $\alpha\in\End_{\mathcal{C}}(X)$ is an isomorphism whenever $f=\alpha f$; a \emph{$\mathcal{D}$-epic} morphism if the induced map $f^*: \Hom_\mathcal{C}(D, X)\to \Hom_\mathcal{C}(D, Y)$ is surjective for any $D\in\mathcal{D}$; and a \emph{right $\mathcal{D}$-approximation} of $Y$ if $X\in \mathcal{D}$ and $f$ is $\mathcal{D}$-epic.  If $f$ is both a right minimal morphism and a right $\mathcal{D}$-approximation of $Y$, then $f$ is called a right \emph{minimal $\mathcal{D}$-approximation} of $Y$. Dually, one defines the notions of $\mathcal{D}$-monic morphisms and left (minimal) $\mathcal{D}$-approximations in $\mathcal{C}$.

A diagram $X \xr{\lambda} Y \xr{\pi} Z$ of morphisms in $\mathcal{C}$ is called a \emph{kernel-cokernel pair} $(\lambda,\pi)$  if $\lambda$ is the kernel of $\pi$ and  $\pi$ is the cokernel of $\lambda$. Given another kernel-cokernel pair $(\lambda',\pi')$ in $\mathcal{C}$, if there are isomorphisms $a:X\to X', b: Y\to Y'$ and $c:Z\to Z'$ in $\mathcal{C}$ such that $\lambda b=a\lambda'$ and $\pi c= b\pi'$, we say that the pairs $(\lambda,\pi)$ and $(\lambda',\pi')$ are \emph{isomorphic}.
Let $\mathcal{S}$ be a collection of kernel-cokernel pairs in $\mathcal{C}$ which is closed under isomorphisms. A morphism $\lambda:X\to Y$ in $\mathcal{C}$ is called an \emph{admissible monomorphism of $\mathcal{S}$} if there is a morphism $\pi:Y\to Z$ in $\mathcal{C}$ such that $(\lambda,\pi)\in \mathcal{S}$. Dually, we define admissible epimorphisms of $\mathcal{S}$. The elements of $\mathcal{S}$ are called \emph{admissible exact sequences} of $\mathcal{S}$. If $\mathcal{C}$ together with $\mathcal{S}$ forms an exact category, we often say that $(\mathcal{C},\mathcal{S})$ is an exact category. We refer to \cite{q1}
for the precise formulation of exact categories.

Let $(\mathcal{C},\mathcal{S})$ and $(\mathcal{C}^\prime,\mathcal{S}^\prime)$ be exact categories. An additive functor $F:\mathcal{C}\to \mathcal{C}^\prime$ is called an \emph{exact functor} if $F$ sends  admissible exact sequences of $\mathcal{S}$ to admissible exact sequences of $\mathcal{S}^\prime$. In this case, we often say that $F: (\mathcal{C},\mathcal{S})\to (\mathcal{C}^\prime,\mathcal{S}^\prime)$ is an exact functor.

Let $\mathcal{D}$ be a full additive subcategory of an exact category $(\mathcal{C},\mathcal{S})$, and let $\tilde{\mathcal{S}}:=\{(\lambda,\pi)\in \mathcal{S} \mid \lambda, \pi \in \mathcal{D}\}$. 
The pair $(\mathcal{D},\tilde{\mathcal{S}})$ is called a \emph{fully exact subcategory} of $(\mathcal{C},\mathcal{S})$ if $(\mathcal{D},\tilde{\mathcal{S}})$ is closed under extensions in $(\mathcal{C},\mathcal{S})$, that is, for an admissible exact sequence $X\to Y\to Z$ of $\mathcal{S}$ with $X,Z \in \mathcal{D}$, we have $Y\in \mathcal{D}$. In this case, the inclusion $(\mathcal{D},\tilde{\mathcal{S}})\subseteq (\mathcal{C},\mathcal{S})$ is a fully faithful exact functor.

Let $(\mathcal{C},\mathcal{S})$ be an exact category. An object $P$ in $\mathcal{C}$ is said to be \emph{$\mathcal{S}$-projective} if, for every admissible epimorphism $f: Y \to Z$ of $\mathcal{S}$, the induced morphism $f^*:\Hom_\mathcal{C}(P,Y)\to \Hom_\mathcal{C}(P,Z)$ is surjective. We say that $(\mathcal{C},\mathcal{S})$ has \emph{enough $\mathcal{S}$-projectives} if, for every object $C \in \mathcal{C}$, there exists an admissible epimorphism $P \to X$ of $\mathcal{S}$ such that $P$ is $\mathcal{S}$-projective. Dually, one defines the notions of \emph{$\mathcal{S}$-injective} objects and exact categories with \emph{enough $\mathcal{S}$-injectives}. We denote by ${\rm Proj}_\mathcal{S}({\mathcal{C}})$ (respectively, ${\rm Inj}_\mathcal{S}({\mathcal{C}})$) the full subcategory of $\mathcal{C}$ consisting of all $\mathcal{S}$-projective (respectively, $\mathcal{S}$-injective) objects.

A \emph{Frobenius category} is an exact category $(\mathcal{C},\mathcal{S})$ which has both enough $\mathcal{S}$-projectives and $\mathcal{S}$-injectives, such that ${\rm Proj}_\mathcal{S}({\mathcal{C}})$ and ${\rm Inj}_\mathcal{S}({\mathcal{C}})$  coincide. In this case, the quotient $\mathcal{C}/{\rm Inj}_\mathcal{S}({\mathcal{C}})$ is a triangulated category \cite[Theorem 2.6, p.16]{h1}. Frobenius categories $(\mathcal{C},\mathcal{S})$ and $(\mathcal{C}^\prime,\mathcal{S}^\prime)$ are \emph{equivalent} if there is an exact functor $F:(\mathcal{C},\mathcal{S})\to (\mathcal{C}^\prime,\mathcal{S}^\prime)$ such that $F:\mathcal{C}\to\mathcal{C}'$ is an equivalence of additive categories and induces an equivalence: $\rm{Proj}_\mathcal{S}(\mathcal{C})\to \rm{Proj}_\mathcal{S'}(\mathcal{C}')$ of additive categories. More generally, if $F:(\mathcal{C},\mathcal{S})\to (\mathcal{C}',\mathcal{S}')$ is an exact functor between Frobenius categories $(\mathcal{C},\mathcal{S})$ and $(\mathcal{C}^\prime,\mathcal{S}^\prime)$ such that $F$ sends objects in ${\rm Inj}_\mathcal{S}({\mathcal{C}})$ to objects in ${\rm Inj}_\mathcal{S'}({\mathcal{C}'})$, then $F$ induces a triangle functor $\underline{F}:\mathcal{C}/{\rm Inj}_\mathcal{S}({\mathcal{C}})\to \mathcal{C}'/{\rm Inj}_\mathcal{S'}({\mathcal{C}'})$ (see \cite[Example 8.1]{k1}).
Consequently, we get the following result.

\begin{Koro}\label{eeite}
Let $(\mathcal{C},\mathcal{S})$ be a Frobenius category and let $(\mathcal{C}^\prime,\mathcal{S}^\prime)$ be a full exact subcategory of $(\mathcal{C},\mathcal{S})$. If $(\mathcal{C}^\prime,\mathcal{S}^\prime)$ is a Frobenius category such that ${\rm Inj}_\mathcal{S'}({\mathcal{C}'})=\mathcal{C}^\prime\cap {\rm Inj}_\mathcal{S}({\mathcal{C}})$, then the inclusion functor $F:(\mathcal{C}^\prime,\mathcal{S}^\prime)\to(\mathcal{C},\mathcal{S})$ induces a fully faithful triangle functor $\underline{F}:\mathcal{C}^\prime/{\rm Inj}_\mathcal{S'}({\mathcal{C}'})\to \mathcal{C}/{\rm Inj}_\mathcal{S}({\mathcal{C}})$.
\end{Koro}

To get intrinsic descriptions of products of additive categories, we observe the following simple fact.
\begin{Lem}\label{acstb}
Let $\mathcal{C}$ be an additive (triangulated) category, and let $\mathcal{C}_1$ and $\mathcal{C}_2$ be full additive (triangulated) subcategories  of $\mathcal{C}$. Suppose $\Hom_\mathcal{C}(C_1,C_2)=0=\Hom_\mathcal{C}(C_2,C_1)$ for all $C_1\in \mathcal{C}_1$ and $C_2\in\mathcal{C}_2$. If, for any object $C\in \mathcal{C}$, there are $C_1\in \mathcal{C}_1$ and $C_2\in \mathcal{C}_2$ such that $C\simeq  C_1\oplus C_2$ in $\mathcal{C}$, then $\mathcal{C}\simeq \mathcal{C}_1\times \mathcal{C}_2$ as additive (triangulated) categories, where $\mathcal{C}_1\times \mathcal{C}_2$ is the product of categories.
\end{Lem}

\subsection{Singularity categories of algebras}\label{sect2.1}

Let $\Lambda$ be an Artin $\mathbb{K}$-algebra over a commutative Artin ring $\mathbb{K}$. By $\rad(\Lambda)$ we denote the Jacobson radical of $\Lambda$. Let $\Lambda^{\opp}$ stand for the opposite algebra of $\Lambda$. We write $\Lambda\modcat$ for the category of all finitely generated $\Lambda$-modules, and $\prj{\Lambda}$ (respectively, $\inj{\Lambda}$) for the full subcategory of $\Lambda\modcat$ consisting of projective (respectively, injective) $\Lambda$-modules.

For a $\Lambda$-module $M$, $\ell_\Lambda(M)$ and $LL(M)$ denote the composition length and Loewy length of $M$, respectively. The\emph{ basic module} of $M$ is denoted by $\mathcal{B}(M)$ which is, by definition, the direct sum of all non-isomorphic indecomposable direct summands of $M$. We write $M = \bigoplus_{i=1}^m M_i,$ where all $M_i$ are indecomposable, and denote $m$ by $\#(M)$. Clearly, $\#(M)$ is uniquely determined by $M$.

A module $M\in \Lambda\modcat$ is called a \emph{generator} (or \emph{cogenerator}) if $\add(_\Lambda\Lambda)\subseteq \add(M)$ (or $\add(D(\Lambda_\Lambda))$ $\subseteq \add(M)$), where $D: \Lambda\modcat\to \Lambda^{\opp}\modcat$ is the usual duality of an Artin algebra, and an \emph{additive generator} if $\Lambda\modcat = \add(M)$. A module $_\Lambda X$ is said to be \emph{$n$-self-orthogonal} if $\Ext_\Lambda^i(X,X)=0$ for all $1\le i\le n$,  and \emph{self-orthogonal} if it is $n$-orthogonal for all $n\ge 1$.

We recall the definition of Gorenstein projective objects.

\begin{Def} \label{gorenstein} Let $\mathcal{A}$ be an abelian category with enough projectives. An object $X\in \mathcal{A}$ is said to be  \emph{Gorenstein projective} if there is an exact sequence
$ \cpx{P}: \cdots \ra P^{-1}\ra P_0\lraf{d^0} P^1\ra P^2\ra \cdots$
of projective objects in $\mathcal{A}$ such that $X\simeq \Ker(d^0)$ and the complex
$\Hom_{\mathcal{A}}(\cpx{P},P)$ is exact for any projective object $P\in \mathcal{A}$.
\end{Def}

Let $\mathcal{A}\gp$ denote the full subcategory of all Gorenstein projective objects of $\mathcal{A}$.
If $\mathcal{A}= \Lambda\modcat$ for an Artin algebra $\Lambda$, then Gorenstein objects in $\mathcal{A}$ are called \emph{Gorenstein projective $\Lambda$-modules}. We write $\gp{\Lambda}$  for the category of all Gorenstein projective modules in $\Lambda\modcat$. Then $\gp{\Lambda}$ contains $\prj{\Lambda}$ and is closed under direct summands, extensions and kernels of epimorphisms. It is a resolving subcategory of $\Lambda\text{-}\mathrm{mod}$. If $\gp{\Lambda}$ has only finitely many non-isomorphic indecomposable modules, then we say that $\Lambda$ is \emph{CM-finite}.

Let $\stmc{\Lambda}:=\Lambda\modcat/\prj{\Lambda}$ be the \emph{stable module category} of $\Lambda$.  Since $\gp{\Lambda}$ contains $\prj{\Lambda}$, we have the stable subcategory $\Stgp{\Lambda}:=\gp{\Lambda}/\prj{\Lambda}$ of $\stmc{\Lambda}$. It is called the \emph{Gorenstein stable category} of $\Lambda$.
We define $$\mathcal{S}:=\{X\xr{\lambda} Y\xr{\pi}Z\mid (\lambda,\pi) \text{ is a kernel-cokernel pair in } \Lambda\modcat \text{ and }X, Y, Z\in \gp{\Lambda}\}.$$ It is well known that $(\gp{\Lambda},\mathcal{S})$ is a Frobenius $\mathbb{K}$-category with ${\rm Inj}_{\mathcal{S}}(\gp{\Lambda})=\prj{\Lambda}$. Thus $\Stgp{\Lambda}$ is a triangulated $\mathbb{K}$-category.

\smallskip
By replacing module categories and projective modules by derived categories and projective complexes, respectively, one gets the so-called singularity categories: Let $\Db{\Lambda}$ stand for the bounded derived category of $\Lambda\modcat$, and let $\Kb{\prj{\Lambda}}$ be the full subcategory of $\Db{\Lambda}$ consisting of the bounded complexes of finitely generated projective $\Lambda$-modules. Then $\Kb{\prj{\Lambda}}$ is a thick triangulated $\mathbb{K}$-subcategory of $\Db{\Lambda}$. Let $\Ds{\Lambda}:=\Db{\Lambda}/\Kb{\prj{\Lambda}}$ be the Verdier quotient of $\Db{\Lambda}$ by $\Kb{\prj{\Lambda}}$. This  triangulated $\mathbb{K}$-category is called the \emph{singularity category} of $\Lambda$.

If an Artin algebra $\Lambda$ is self-injective, then $\Ds{\Lambda}$ and $\stmc{\Lambda}$ are equivalent as triangulated $\mathbb{K}$-categories. More generally, if $\Lambda$ is a Gorenstein algebra (that is, the injective dimensions of ${}_\Lambda\Lambda$ and $\Lambda_{\Lambda}$ are finite), then  $\Ds{\Lambda}\simeq \Stgp{\Lambda}$ as triangulated $\mathbb{K}$-categories (see \cite{b2} or \cite{h2}).

\begin{Def}
Artin $\mathbb{K}$-algebras $\Lambda$ and $\Gamma$ are said to be

$(1)$ \emph{Stably equivalent} if their stable module categories $\Lambda\stmc$ and $\Gamma\stmc$ are equivalent as $\mathbb{K}$-categories. An equivalence $F: \stmc{\Lambda} \ra  \Gamma\stmc$ of $\,\mathbb{K}$-categories is called a \emph{stable equivalence} between $\Lambda$ and $\Gamma$.

$(2)$ \emph{Singularly equivalent} if their singularity categories $\Ds{\Lambda}$ and $\Ds{\Gamma}$ are equivalent as triangulated $\mathbb{K}$-categories. An equivalence $F:\Ds{\Lambda}\rightarrow\Ds{\Gamma}$ of triangulated $\mathbb{K}$-categories is called a \emph{singular equivalence} between $\Lambda$ and $\Gamma$.
\end{Def}

The above algebraical approach to singularity categories and singular equivalences was introduced and studied in \cite{b2}, while a geometrical approach was formulated and investigated independently in \cite{orlov}. Since then singularity categories and singular equivalences have been intensively studied.

\subsection{Basic facts on modules over quotients of polynomial algebras\label{sect2.3}}
In this subsection we prove some lemmas of modules over quotients of the polynomial algebra $R[x]$.

Let $\mathbb{Z}_{>0}$ be the set of positive integers. For $n\in\mathbb{Z}_{>0}$, let $[n]:=\{1,2, \cdots,n\}$ and $\Sigma_n$ be the symmetric group of permutations of $[n]$. We write the image of $i\in [n]$ under $\sigma\in \Sigma_n$ as $(i)\sigma$. The cardinality of a set $S$ is denoted by $|S|$.

In the rest of this subsection, we fix an irreducible polynomial $f(x)\in R[x]$ of degree $u$ and set $A:=R[x]/(f(x)^n)$ for $n\geq 1$.
Then $A$ is a local, commutative, symmetric, Nakayama algebra 
with $n$ (non-isomorphic) indecomposable modules $M(i)$ for $i\in [n]$, where $M(i)$ equals $R[x]/(f(x)^i)\in A\modcat$ for $i\in [n]$. We set $M(0)=0$. Clearly, $M(i)\simeq A/\rad^i(A)$ as $A$-modules for $0\le i\le n.$

For $0\le i\le j \le n$, we denote by
$$f_{i,j}: M(i)\ra M(j), \; x^k+(f(x)^i)\mapsto f(x)^{j-i}\cdot x^k+(f(x)^j) \mbox{ for } 0\le k\le ui-1,\ \mbox{ and } $$ $$g_{j,i}: M(j)\ra M(i), \; x^k+(f(x)^j)\mapsto x^k+(f(x)^i) \mbox{ for } 0\le k\le uj-1$$ the canonical injective and surjective homomorphisms, respectively. Clearly, $f_{i,j}g_{j,k}=0$ if $i+k \le j.$ 

\begin{Lem}\label{ftapi}
Let $i, j, k\in [n]$ with $j, k\ge i$, and let $f:M(i)\to M(k), g:M(k)\to M(i)$ and $h:M(j-i)\to M(k-i)$ be homomorphisms of $A$-modules.

$(1)$ If $k\ge j$, then $f$ and $g$ factorize through $M(j)$.

$(2)$ There exists a homomorphism $\tilde{h}:M(j)\to M(k)$ such that $\tilde{h}g_{k,k-i}=g_{j,j-i}h$.
\end{Lem}

{\it Proof.} $(1)$ Due to $j \ge i$, we have a canonical surjective homomorphism $g_{j,i}:M(j)\to M(i)$. Due to $k\ge j$, the homomorphisms $g$ and $g_{j,i}$ can be regarded as homomorphisms of $R[x]/(f(x)^k)$-modules. Since $M(k)$ is a projective $R[x]/(f(x)^k)$-module, there exists a homomorphism $g_1:M(k)\to M(j)$ of $R[x]/(f(x)^k)$-modules such that $g_1 g_{j,i}=g$. This shows that $g$ factorizes through $M(j)$. Similarly, we show that $f$ factorizes through $M(j)$.

(2) Assume $j\ge k$. Then $g_{k,k-i}$, $g_{j,j-i}$ and $h$ can be viewed as homomorphisms of $R[x]/(f(x)^j)$-modules. Since $M(j)$ is a projective $R[x]/(f(x)^j)$-module and $g_{k,k-i}$ is an epimorphism, there exists a homomorphism $\tilde{h}:M(j)\to M(k)$ such that $\tilde{h}g_{k,k-i}=g_{j,j-i}h$.

Now we assume $j<k$. Since the map
$$\varphi:R[x]/(f(x)^{j-i})\lra \Hom_A(R[x]/(f(x)^{j-i}),R[x]/(f(x)^{k-i})), $$ $$ \alpha(x)+f(x)^{j-i}\mapsto [x^s+(f(x)^{j-i})\mapsto \alpha(x)\cdot f(x)^{k-j}\cdot x^s+(f(x)^{k-i})]$$for $ \alpha(x)\in R[x]$ and $0\le s\le u(j-i)-1,$ is an isomorphism of $A$-modules, we may assume that there exists a polynomial $\alpha(x)\in R[x]$ such that $h=(\alpha(x)+f(x)^{j-i})\varphi$. The surjective map $g_{j,j-i}$ supplies a polynomial $\tilde{g}(x)\in R[x]$ such that $(\tilde{g}(x)+f(x)^{j})g_{j,j-i}=\alpha(x)+f(x)^{j-i}$. Define $$\tilde{h}:R[x]/(f(x)^j)\lra R[x]/(f(x)^k), \; x^t+(f(x)^{j})\mapsto \tilde{g}(x)\cdot f(x)^{k-j}\cdot x^t+(f(x)^{k}),\ \text{for}\ 0\le t\le uj-1.$$ Then one can check that $\tilde{h}$ is a homomorphism of $A$-modules and $\tilde{h}g_{k,k-i}=g_{j,j-i}h$. $\square$

\medskip
Recall that two multisets $\{\{n_1,\cdots,n_s\}\}$ and $\{\{m_1,\cdots,m_t\}\}$ are equal if and only if $s=t$ and there exists a permutation $\sigma\in \Sigma_s$ such that $(m_1,\cdots,m_s)^\sigma:=(m_{(1)\sigma},\cdots,m_{(s)\sigma})=(n_1,\cdots,n_s)$ in $\mathbb{N}^s.$

\begin{Lem}\label{ecai}
Let $g(x)$ be an irreducible polynomial in $R[x]$ and $B:=R[x]/(g(x)^m)$ for some $m\in\N ^{+}$. Suppose that $I:=\{\{n_1,\cdots,n_s\}\}$
and $J:=\{\{m_1,\cdots,m_t\}\}$ are nonempty multisets with $n_i\in [n]$ for $i\in [s]$ and $m_j\in [m]$ for $j\in[t]$. Let $n'$ and $m'$ be the maximal elements of $I$ and $J$, respectively. 
Let ${}_AM:=\bigoplus_{i\in [s]}R[x]/(f(x)^{n_i})$ and ${}_BN:=\bigoplus_{j\in [t]}R[x]/(g(x)^{m_j})$. Then the following  are equivalent.

$(1)$ $\End_A(M)\simeq \End_B(N)$ as $R$-algebras.

$(2)$ $I=J$ and $R[x]/(f(x)^{n'})\simeq R[x]/(g(x)^{m'})$ as $R$-algebras.
\end{Lem}

{\it Proof.} We may assume $n'=n$ and $m'=m$. Otherwise we replace $A$ and $B$ by $A'=R[x]/(f(x)^{n'})$ and $B':=R[x]/(g(x)^{m'})$, respectively, and regard $M$ as an $A'$-module and $N$ as a $B'$-module. In this case, we have $\End_A(M)=\End_{A'}(M)$ and $\End_B(N)=\End_{B'}(N)$.

The implication (2) $\Ra$ (1) is clear. We prove  (1) $\Ra$ (2). Indeed, let $\Lambda:=\End_A(M)$ and $\Gamma:=\End_B(N)$. As left modules, $_{\Lambda}\Lambda= \bigoplus_{i\in [s]}\Hom_A(M,R[x]/(f(x)^{n_i})$ and  $_{\Gamma}\Gamma= \bigoplus_{j\in [t]}\Hom_B(N,R[x]/(g(x)^{m_j})$.  The summands $_\Lambda\Hom_A(M,R[x]/(f(x)^{n_i})$ and $_\Gamma\Hom_B(N,R[x]/(g(x)^{m_j})$ are indecomposable projective modules for $i\in [s]$ and $j\in [t]$. Since $\Lambda$ and $\Gamma$ are isomorphic $R$-algebras, we have $s=t$, and there exists  $\sigma \in \Sigma_{s}$ such that $\End_\Lambda(\Hom_A(M,R[x]/(f(x)^{n_i}))\simeq \End_\Gamma(\Hom_B(N,R[x]/(g(x)^{m_{(i)\sigma}}))$ as $R$-algebras. This implies that $$R[x]/(f(x)^{n_i})\simeq \End_\Lambda(\Hom_A(M,R[x]/(f(x)^{n_i}))\simeq \End_\Gamma(\Hom_B(N,R[x]/(g(x)^{m_{(i)\sigma}}))\simeq R[x]/(g(x)^{m_{(i)\sigma}})$$ as $R$-algebras. Thus $n_i=LL(R[x]/(f(x)^{n_i}))=LL(R[x]/(g(x)^{m_{(i)\sigma}}))=m_{(i)\sigma}$ for $i\in [s]$. This implies that $I=J$ and $A\simeq B$ as $R$-algebras. $\square$

\begin{Lem}\label{exies}
Let $\Lambda$ be an Artin algebra, $\eta: 0\to X\xr{f} Y\xr{g} Z\to 0$ an exact sequence in $\Lambda\modcat$ and $i\in \mathbb{N}$. If $g$ is $\add\big(\Lambda/\rad^i(\Lambda)\big)$-epic, then $\eta$ induces two exact sequences: $$0\lra \soc^i(X)\lraf{f^\prime} \soc^i(Y)\lraf{g^\prime} \soc^i(Z)\lra 0 \; \mbox{ and } $$ $$0\lra X/\soc^i(X)\lraf{f^\wedge} Y/\soc^i(Y)\lraf{g^\wedge} Z/\soc^i(Z)\lra 0,$$ where $f^\prime$ and $g^\prime$ are the restrictions of $f$ and $g$, respectively, and where $f^\wedge$ and $g^\wedge$ are the induced homomorphisms of $f$ and $g$, respectively.
\end{Lem}

{\it Proof.} Given $M\in \Lambda\modcat$ and $i\in \mathbb{N}$, the map $\rho^i_M:\Hom_\Lambda(\Lambda/\rad^i(\Lambda),M)\simeq \soc^i(M)$, $f\mapsto (1)f$ for $f\in \Hom_\Lambda(\Lambda/\rad^i(\Lambda),M)$, is an isomorphism of $\Lambda$-modules, and any homomorphism $h:M\to M'$ of $\Lambda$-modules restricts to a homomorphism $h^\prime:\soc^i(M)\to \soc^i(M')$ of $\Lambda$-modules.
We form the following commutative diagram of $0$-sequences in $\Lambda\modcat$:
$$\xymatrix@R=0.5cm{
 0\ar[r] & \Hom_\Lambda(\Lambda/\rad^i(\Lambda),X)\ar[r]^{f^*}\ar[d]_{\rho^i_X}^{\simeq} &  \Hom_\Lambda(\Lambda/\rad^i(\Lambda),Y)\ar[r]^{g^*}\ar[d]_{\rho^i_Y}^{\simeq} &\Hom_\Lambda(\Lambda/\rad^i(\Lambda),Z)\ar[r]\ar[d]_{\rho^i_Z}^{\simeq} & 0\\
    0\ar[r] & \soc^i(X) \ar[r]^{f^\prime}\ar@{^{(}->}[d] & \soc^i(Y) \ar[r]^{g^\prime}\ar@{^{(}->}[d] &\soc^i(Z) \ar[r]\ar@{^{(}->}[d] & 0\\
    0\ar[r] & X\ar[r]_f \ar@{>>}[d] & Y\ar[r]_g\ar@{>>}[d] & Z \ar[r] \ar@{>>}[d]& 0\\
    0\ar[r] & X/\soc^i(X)\ar[r]_{f^\wedge} & Y/\soc^i(Y)\ar[r]_{g^\wedge} &Z/\soc^i(Z)\ar[r] & 0.\\
}$$
As $g$ is $\add(\Lambda/\rad^i(\Lambda))$-epic, the map $g^*$ is surjective. Thus the top row is exact, and therefore the second row is exact. Since all rows and columns are exact except the last row, it follows from $3\times 3$ Lemma that the last row is exact, too. $\square$

\smallskip
Suppose that $\{S_1, \cdots, S_n\}$ is a complete set of non-isomorphic simple $\Lambda$-modules and $P_i$ is the projective cover of $S_i$ for $i\in [n]$. For  $i, j\in [n]$, let $[P_i:S_j]$ denote the number of composition factors of $P_i$ that are isomorphic to $S_j$. Then $[P_i:S_j]= \ell_{\End_\Lambda(P_j)}\big(\Hom_\Lambda(P_j,P_i)\big)$ for $i, j\in [n]$. The matrix $C_\Lambda=([P_i:S_j])_{i,j\in[n]}\in M_n(\mathbb{N})$ is called the \emph{Cartan matrix of $\Lambda$}.
The determinant of $C_\Lambda$ is called the \emph{Cartan determinant} of $\Lambda$.

The following lemma is an analogy of \cite[Section 8]{dps}, except we remove the assumption on the ground field. For the convienence of the reader, we include here a proof.

\begin{Lem}\label{cmoe}
Let $T=\{t_1,\cdots,t_s\}$ be a subset of $[n]$ with $t_1>t_2>\cdots>t_s$, $_AM:=\bigoplus_{i=1}^sM(t_i)$ and $\Gamma:=\End_A(M)$. Then the Cartan matrix $C_\Gamma$ of $\Gamma$ is
$$\begin{pmatrix}
t_1 & t_2  & \cdots & t_{s} \\
t_2 & t_2  & \cdots & t_{s} \\
\vdots & \vdots  & \ddots & \vdots \\
t_{s} & t_{s}  & \cdots & t_{s}
\end{pmatrix}.$$
In particular, $C_\Gamma$ is a positive definite matrix and $\det(C_\Gamma)= t_s\prod_{i=1}^{s-1} (t_i - t_{i+1})$. Moreover, $\det(C_\Gamma)=1$ if and only if $M$ is an additive generator for $R[x]/(f(x)^{t_1})\modcat$.
\end{Lem}

{\it Proof.} We may assume $t_1=n$. Otherwise we consider the algebra $A':=R[x]/(f(x)^{t_1})$ instead of $A$. Note that $M$ can be regarded as an $A'$-module and that the canonical surjective homomorphism $A\to A'$ implies $\End_{A'}(M)=\Gamma$.

The set $\{P_i:=\Hom_A(M,M(t_i))\mid i\in [s]\}$ is a complete set of non-isomorphic indecomposable projective $\Gamma$-modules. Since $\Hom_A(M,-):\add(M)\to \prj{\Gamma}$ is an equivalence of additive categories, we have $\Hom_\Gamma(P_j,P_i)\simeq \Hom_A(M(t_j),M(t_i))\simeq R[x]/(f(x)^{m_{ji}})$ as $R[x]/(f(x)^{t_j})$-modules for $i, j\in[s]$ with $m_{ji}:=\min\{t_j,t_i\}$.  Note that $R[x]/(f(x)^{m_{ji}})$-module structure on $\Hom_\Gamma(P_j,P_i)$ is induced by the $R$-algebra isomorphism $\End_{\Gamma}(P_j)\simeq R[x]/(f(x)^{t_j})$. Thus $\ell_{\End_\Gamma(P_j)}(\Hom_{\Gamma}(P_j,P_i))=\ell_{R[x]/(f(x)^{t_j})}(R[x]/(f(x)^{m_{ji}}))=m_{ji}$ for $i, j\in [s]$. Consequently, the Cartan matrix $C_\Gamma$ of $\Gamma$ is just the matrix in Lemma \ref{cmoe}.
Thus $C_\Gamma$ is positive definite and $\det(C_\Gamma)= t_s\prod_{i=1}^{s-1} (t_i - t_{i+1})$. Clearly, $\det(C_\Gamma)=1$ if and only if $t_s=1$ and $t_i - t_{i+1}=1$ for $i\in[s-1]$ if and only if $M$ is an additive generator for $R[x]/(f(x)^{t_1})\modcat$. $\square$

\medskip
We quote the following two results from \cite{lx1,lx2}.

\begin{Lem}{\rm \cite[Lemma 2.11]{lx1}}\label{lift-exact}
Given $\{a,b,c,d\}\subseteq \{0,1,\cdots,n\}$ with $b < a< c$, $b< d< c$ and $a+d=b+c$, and $X\in A\modcat$ which does not have  indecomposable direct summands $N$ with $b<\ell_A(N)<c$, let $_AY:={}_AX\oplus M(b)\oplus M(c)$. Then there is an exact sequence $0\ra M(a)\lraf{f} M(b)\oplus M(c)\lraf{g} M(d)\ra 0$ in $A\modcat$, where $f=(-g_{a,b},f_{a,c})$ is a left minimal $\add(Y)$-approximation of $M(a)$ and $g=\left(\begin{smallmatrix} f_{b,d} \\ g_{c,d} \end{smallmatrix}\right)$ is a right minimal $\add(Y)$-approximation of $M(d)$.
\end{Lem}

\begin{Lem} {\rm \cite[Lemma 2.10]{lx2}}\label{St-i} Let $g(x)$ be an irreducible polynomial in $R[x]$ and $B:=R[x]/(g(x)^m)$ for $m\ge 2$.
Then the following are equivalent:

$(1)$ $\stmc{A}\simeq \stmc{B}$ as triangulated $R$-categories.

$(2)$ $\stmc{A}\simeq \stmc{B}$ as $R$-categories.

$(3)$ $n=m$ and $A/\rad(A)\simeq B/\rad(B)$ as $R$-algebras.

$(4)$ $A\simeq B$ as $R$-algebras.
\end{Lem}

From Lemma \ref{lift-exact}, we have the following.
\begin{Koro}\label{esam}
Let $a,b\in \{0,1,\cdots,n\}$ with $a<b$, and let $M:=\bigoplus_{i=a}^b M(i)$, $N:=M(a)\oplus M(b)$ and $K:= \bigoplus_{i=0}^aM(i)\oplus\bigoplus_{j=b}^{n}M(j)$ be $A$-modules. Then

$(1)$ For any $X\in \add(M)$, there exist two exact sequences in $A\modcat$:
$$0 \lra X \lraf{f} Y \lraf{g} Z \lra 0 \; \text{ and } \; 0\lra Z'\lraf{f'} Y'\lraf{g'}X \lra 0$$ with $Y, Y'\in \add(N)$ and $Z, Z'\in \add(M)$, such that $f$ and $f'$ are left minimal $\add(K)$-approximations of $X$ and $Z'$, respectively, and that $g$ and $g'$ are right minimal $\add(K)$-approximations of $Z$ and $Z'$, respectively.

$(2)$ For any homomorphism $h:X\to Y$ in $A\modcat$, then

$\quad (i)$ If $Y\in \add(M)$, then $h$ is $\add(K)$-epic if and only if $h$ is $\add(N)$-epic.

$\quad (ii)$ If $X\in \add(M)$, then $h$ is $\add(K)$-monic if and only if $h$ is $\add(N)$-monic.
\end{Koro}

{\it Proof.} (1) Indeed, it is enough to assume that $X$ is indecomposable in $\add(M)$, that is, $X=M(i)$ for some $a\le i\le b$. If $i=a$ or $b$, then the trivial exact sequences $0 \to X \xr{1_X} X \to 0 \to 0 $ satisfy the required conditions.

Assume that $i\ne a$ and $i\ne b$. Then $a<i<b, a<a+b-i<b$ and $a+b=i+(a+b-i)$. It follows from Lemma \ref{lift-exact} that there exist an exact sequence in $A\modcat$: $0\lra M(i) \lraf{f} M(a)\oplus M(b)\lraf{g} M(a+b-i)\lra 0$ such that $f$ and $g$ are left and right minimal $\add(K)$-approximations, respectively. 
Similarly, we have an exact sequence $0\lra Z'\lraf{f'} Y'\lraf{g'}X \lra 0$ satisfying the required conditions. This shows (1).

(2) Since (ii) is dual to (i), we only prove (i). Suppose $h$ is $\add(K)$-epic. Since $N\in \add(K)$, it follows that $h$ is $\add(N)$-epic. Conversely, we show that if $h$ is $\add(N)$-epic, then $h$ is $\add(K)$-epic.

Let $h':K'\to Y$ be a homomorphism in $A\modcat$ with $K'\in\add(K)$. Since $Y\in\add(M)$, there is a right $\add(K)$-approximation $h'':K''\to Y$ of $Y$ with $K''\in\add(N)$ by (1). Since $h$ is $\add(N)$-epic, there is a homomorphism $\phi:K''\to X$ such that $\phi h=h''$. Since $h''$ is a right $\add(K)$-approximation and $K'\in\add(K)$, there exists a homomorphism $\psi:K'\to K''$ such that $\psi h''=h'$. Hence $\psi \phi h=\psi h''=h'$, and therefore $h$ is $\add(K)$-epic. $\square$

\subsection{Centralizer algebras of matrices}
Centralizer algebras of matrices can realize any finite-dimensional algebras over fields. So it is worthy to investigate finite-dimensional algebras from the view point of the entralizer algebras of matrices. Let us recall some definitions and basic results on centralizer algebras of matrices.

Given a field $R$ and a natural number $n\in \mathbb{Z}_{>0}$, let $M_n(R)$ be the full $n\times n$ matrix algebra over $R$ with the identity matrix $I_n$. Let $e_{ij}$, $1\le i,j\le n$, be the matrix units, and $J_n(\lambda)\in M_n(R)$ be the Jordan block matrix with the eigenvalue $\lambda\in R$.

For a nonempty subset $X$ of $M_n(R)$, the \emph{centralizer algebra} $S_n(X,R)$ of $X$ in $M_n(R)$ is defined by $$S_n(X,R):=\{a\in M_n(R)\mid ax=xa,\; \forall \; x\in X\}.$$

If $X$ is a finite set, Brenner reduced the study of $S_n(X,R)$ to the case that $X$ consists of only two elements \cite[Lemma 1]{brenner1972}. Furthermore, Brenner showed in \cite[Lemma 2]{brenner1972} that every finite-dimensional algebra over a field is isomorphic to the centralizer algebra of \textbf{two} matrices. Thus it is natural and fundamental to study first the centralizer algebra of a single matrix. This might be an important step to understanding arbitrary finite-dimensional algebras and related topics. For simplicity, we write $S_n(c,R)$ for $S_n(\{c\},R)$.
By a \emph{centralizer matrix algebra} we always mean an algebra of the form $S_n(c,R)$.

We have the following property of centralizer matrix algebras.
\begin{Lem}{\rm \cite[Lemma 3.2]{lx1}}\label{cmat}
For $c\in M_n(R)$, there are the isomorphisms of $R$-algebras: $$S_n(c,R)\simeq S_n(c^{tr},R)\simeq S_n(c,R)^{\opp}\simeq \End_{R[c]}(R^n),$$ where $R[c]$ is the subalgebra of $M_n(R)$ generated by $c$, and where $c^{tr}$ is the transpose of $c$.
\end{Lem}

For a matric $c\in M_n(R)$, we denote by
$m_c(x)$ and $\chi_c(x)$ the minimal and characteristic polynomials of a matric $c$ over $R$, respectively.

Next, we point out a connection of centralizer matrix algebras with the algebras $\Gamma$ in Lemma \ref{cmoe}.

Given a monic polynomial $g(x)=x^n+a_{n-1}x^{n-1}+\cdots+a_1x+a_0\in R[x]$. The \emph{companion matrix of $g(x)$} is defined by
$$C[g(x)]:=\begin{pmatrix}
0 &  \cdots & 0 & -a_0 \\
1 &  \cdots & 0 & -a_1 \\
\vdots  & \ddots & \vdots & \vdots \\
0 &  \cdots & 1 & -a_{n-1}
\end{pmatrix}\in M_{n}(R).$$
Then $m_{C[g(x)]}(x)=\chi_{C[g(x)]}(x)=g(x)$.

\begin{Lem}\label{eogoqpicma}
Let $T=\{t_1,\cdots,t_s\}$ be a subset of $[n]$ with $t_1>t_2>\cdots>t_s$, $m:=\sum_{i=1}^sut_i$, $M:=\bigoplus_{i=1}^sM(t_i)\in A\modcat$ and $\Gamma:=\End_A(M)$. Then there is a matrix $c\in M_m(R)$ such that $\Gamma \simeq S_m(c,R)$ as $R$-algebras.
\end{Lem}

{\it Proof.} We write $f(x)=a_ux^u+\cdots+a_1x+a_0\in R[x]$ and consider $A':=R[x]/((a_u^{-1}f(x))^{t_1})$. Then there is a surjective homomorphism $\pi: A\to A'$ of $R$-algebras with $\Ker(\pi)=(f(x)^{t_1})/(f(x)^n)$, and $_AM$ can be regarded as an $A'$-module. Thus  $\End_{A'}(M)=\End_A(M)$. Hence we may assume $t_1=n$ and $a_u=1$.

Now, we consider the diagonal block matrix $c:=C[f(x)^{t_1}]\oplus \cdots\oplus C[f(x)^{t_s}]\in M_m(R)$. Then $m_c(x)=f(x)^{t_1}$, $R[c]\simeq R[x]/(m_c(x))\simeq R[x]/(f(x)^{t_1})= A$ as $R$-algebras, and $R^m\simeq \bigoplus_{i=1}^sR[x]/(f(x)^{t_i})$ as $R[c]$-modules (see \cite[Chapter 4]{c1}). Thus $R^m\simeq M$ as $R[c]$-modules. It follows from Lemma \ref{cmat} that $\Gamma=\End_A(M)\simeq \End_{R[c]}(R^m)\simeq S_m(c,R)$ as $R$-algebras. $\square$

\subsection{Basics on number theory\label{number}}
In this subsection we assume that $p=0$ or $p>0$ is a prime number. Let us recall a few results on number theory for our later proofs.

\begin{Lem}\label{slopop}
Suppose $p>0$ be a prime number and $i,j,k,\ell\in\mathbb{N}$ with $j>k$ and $\ell>i$. Then

$(1)$  $p^j-p^k=p^\ell-p^i$ if and only if $j=\ell$ and $k=i$.

$(2)$  $p^i-1=p^k$ if and only if $p=2, i=1$ and $k=0$. 

$(3)$  $p^{j}-p^{k}=p^i$ if and only if $p=2$ and $i=k=j-1$. In particular, $p^{j}-p^{k}=1$ if and only if $p=2, j = 1$, and $k = 0$.
\end{Lem}

Recall that, for a prime $p>0$ and $n\in \mathbb{Z}_{>0}$, we denote by $\nu_p(n)$ the maximal power index $m$ such that $p^m\mid n$. For convenience, we define $\nu_p(n)=0$ for $p=0$ and all $n\in \mathbb{Z}_{>0}$.

\begin{Lem} \label{no-2}

$(1)$ For $p\ge 0$ and $T:=\{ n_1, n_2, \cdots, n_t\}\subset\mathbb{Z}_{>0}$ with $t\ge 2$ and $\nu_p(n_i)=0$ for all $i\in[t]$. If $n_t\nmid n_j$ for all $j\in [t-1]$, then there is an irreducible factor $f(x)$ of $x^{n_t}-1$ such that $f(x)\nmid x^{n_j}-1$ for all $j\in[t-1]$.

$(2)$ For $n, m\in \mathbb{Z}_{>0}$, $n\mid m$ in $\mathbb{N}$ if and only if $x^n-1 \mid x^m-1$ in $R[x]$.
\end{Lem}

For a subset  $T:=\{ m_1, m_2, \cdots, m_s\}$ of $\mathbb{Z}_{>0}$ with $m_1>m_2>\cdots >m_s$, we associate a set $\mathcal{J}_T:=\{m_1,m_1-m_2,\cdots,m_1-m_s\}$ and a multiset $\mathcal{H}_T:= \{\{m_1-m_2,\cdots,m_{s-1}-m_s,m_s\}\}$. Let $\mathcal{D}_T$ be the multiset obtained from $\mathcal{H}_T$ by removing all $1$'s.

\begin{Lem} \label{no-1} For subsets $T:=\{ m_1, m_2, \cdots, m_s\}$ and $H:=\{n_1, n_2, \cdots, n_t\}$ of $\mathbb{Z}_{>0}$ with $m_1>m_2>\cdots >m_s$ and $n_1>n_2>\cdots>n_t$, if $H=\mathcal{J}_T$, then $s=t$, $T=\mathcal{J}_H$ and $\mathcal{H}_T=\mathcal{H}_H$.
\end{Lem}

Let $\#_s(I)$ denote the multiplicity of a natural number $s$ in a multiset $I$. By Lemma \ref{slopop}, we have the following lemma.

\begin{Lem}\label{preofpp}
Suppose $p>0$ and $m, t, t'\in \mathbb{N}$. Let $T:=\{ p^{m_1}, p^{m_2}, \cdots, p^{m_\ell}\}$ with integers $m_1>m_2>\cdots >m_\ell\ge 0$, and $T^+:=T\cup \{p^m\}$.

$(1)$ If $m>t$, we have $\#_{p^m-p^t}(\mathcal{D}_{T^+})\ge \#_{p^m-p^t}(\mathcal{D}_{T})$ and $\#_{p^m}(\mathcal{D}_{T^+})\ge \#_{p^m}(\mathcal{D}_{T})$. Moreover, if $m>m_\ell$ and $p^m-p^{m_k}>1$, then $\#_{p^m-p^{m_k}}(\mathcal{D}_{T^+})>\#_{p^m-p^{m_k}}(\mathcal{D}_{T})$, where $k:=\min\{\,i\in[\ell]\mid m\ge m_i\,\}$.

$(2)$ If $t>m$, then $\#_{p^{t}}(\mathcal{D}_{T})\ge \#_{p^{t}}(\mathcal{D}_{T^+})$. Moreover, if $m_\ell>m$, then $\#_{p^{m_\ell}}(\mathcal{D}_{T})> \#_{p^{m_\ell}}(\mathcal{D}_{T^+})$.

$(3)$ If $t>m>t'$, then $\#_{p^{t}-p^{t'}}(\mathcal{D}_{T})\ge \#_{p^{t}-p^{t'}}(\mathcal{D}_{T^+})$. Moreover, if $\ell\ge 2$ and $m_i>m>m_{i+1}$ for some $i\in [\ell-1]$, then $\#_{p^{m_i}-p^{m_{i+1}}}(\mathcal{D}_{T})$ $ > \#_{p^{m_i}-p^{m_{i+1}}}(\mathcal{D}_{T^+})$.
\end{Lem}

\section{New types of equivalence relations on matrices\label{sect3}}
In this section, we introduce new equivalence relations on all square matrices and compare them with other known equivalence relations.

Given a matrix $c\in M_n(R)$, we denote by
$\tilde{\mathcal{E}}_c$ the multiset of all elementary divisors of $c$, and $\mathcal{E}_c$ the set of elementary divisors of $c$, which is obtained from $\tilde{\mathcal{E}}_c$ by removing duplicate elements. Let
$$\mathcal{M}_c:=\{f(x)\in \mathcal{E}_c\mid f(x) \mbox{ is maximal with respect to the polynomial divisibility} \le\},$$ where $f(x)\le g(x)$ means $f(x)\mid g(x)$ for polynomials $f(x),g(x)\in R[x]$ of positive degree. Let $m_c(x)$ denote the minimal polynomial of $c$ over $R$ and let $d_1(x),\cdots, d_r(x)$ be invariant factors of $c$ with $d_i(x)\mid d_{i+1}(x),$ $1\le i \le r-1$. Then $\mathcal{M}_c$ is determined completely by $d_r(x)=m_c(x)$.

Let $\mathcal{R}_c:=\{f(x)\in \mathcal{M}_c\mid f(x) \mbox{ is reducible}\}$ be the set of all reducible maximal divisors of $c$.

For $f(x)\in \mathcal{M}_c$, let $\tilde{P}_c(f(x))$ be the multiset of \emph{power indices} of $f(x)$ in $\tilde{\mathcal{E}}_c$, defined by
$$\tilde{P}_c(f(x)):=\{\{i\ge 1 \mid  \exists \mbox{ irreducible polynomial } p(x) \mbox{ such that } p(x) \mbox{ divides } f(x), p(x)^i\in \tilde{\mathcal{E}}_c \}\},$$ and let $P_c(f(x))$ be the set of \emph{power indices} of $f(x)$ in $\mathcal{E}_c$, it is obtained from $\tilde{P}_c(f(x))$ by deleting duplicate elements.

Let $\mathcal{I}_c:=\big\{f(x)\in R[x] \mid f(x) \mbox{ is irreducible such that } f(x)^i\in \mathcal{M}_c \text{ for some } i \text{ with } |P_c(f(x)^i)|\neq \max\{j\in P_c(f(x)^i)\}\big\}$. For an irreducible polynomial $g(x)\in R[x]$, $P_c(g(x))$ is defined by
$P_c(g(x)):=P_c(g(x)^i)$ if $g(x)^i\in \mathcal{M}_c$ for some
$i\in \mathbb{N}$, and $P_c(g(x)) :=\emptyset$ otherwise.

Now, we define a relation $\sim$ on $\mathcal{I}_c$: For $f(x), g(x)\in \mathcal{I}_c$, we write $f(x)\sim g(x)$ provided $R[x]/(f(x))\simeq R[x]/(g(x))$ as $R$-algebras. This equivalence relation on $\mathcal{I}_c$ gives rise to a partition of  $\mathcal{I}_c$ into its equivalence classes $\mathcal{I}_{c,1},\cdots, \mathcal{I}_{c,r_c}$, where $r_c$ is the number of the equivalence classes. For $i\in [r_c]$, let $$\mathcal{D}_{c,i}:=\bigcup_{f(x)\in \mathcal{I}_{c,i}}\mathcal{D}_{P_c(f(x))} \mbox{ and } Q_{c,i}:=R[x]/(p_{c,i}(x)),$$
where $p_{c,i}(x)$ is a fixed representatives of the equivalence class $\mathcal{I}_{c,i}$. Note that $\mathcal{D}_{c,i}$ is a union of multisets and that $Q_{c,i}$ is independent of the choice of the representatives, up to isomorphism of $R$-algebras.

We define $\mathcal{U}_c:=\bigcup_{g(x)\in \mathcal{M}_c}\mathcal{D}_{P_c(g(x))}$, where the union is taken in the sense of multisets. Clearly, $\mathcal{U}_c=\bigcup_{f(x)\in \mathcal{I}_c}\mathcal{D}_{P_c(f(x))}$.

\smallskip
Now, we introduce two new types of equivalence relations on all square matrices over a field.

\begin{Def}\label{newequrel}
Two matrices $c\in M_n(R)$ and $d\in M_m(R)$ are said to be

$(1)$ \emph{$I$-equivalent} if there exists a bijection $\pi:\mathcal{M}_c\ra\mathcal{M}_d$, such that $R[x]/(f(x)) \simeq R[x]/((f(x))\pi)$ as $R$-algebras and ${\tilde{P}_c(f(x))} = {\tilde{P}_d((f(x))\pi)}$ for all $f(x)\in \mathcal{M}_c$. In this case, we simply write $c\stackrel{I}\sim d$.

$(2)$  \emph{$Sg$-equivalent} if $r_c=r_d$ and there exists a permutation $\sigma\in \Sigma_{r_c}$ such that $Q_{c,i}\simeq Q_{d,(i)\sigma}$ as $R$-algebras and $\mathcal{D}_{c,i}=\mathcal{D}_{d,(i)\sigma}$ for all $i\in [r_c]$. In this case, we simply write $c\stackrel{Sg}\sim d$.
\end{Def}

Clearly, $c\stackrel{I}\sim d$ and $c\stackrel{Sg}\sim d$ are equivalence relations on the set of all square matrices over $R$.
If $c\stackrel{I}\sim d$, then $\deg m_c(x)=\deg m_d(x)$ and $n=\deg \chi_c(x)=\deg \chi_d(x)=m$. If $c\stackrel{Sg}\sim d$, then $\mathcal{U}_c=\mathcal{U}_d$. In this case, $\mathcal{I}_c=\emptyset$ if and only if $\mathcal{I}_d=\emptyset$.

The two equivalence relations are different from the equivalence relations  introduced in \cite[Definition 3.1]{lx1} and \cite[Definition 3.1]{lx2}. We recall these relations right now.

\begin{Def}{\rm \cite{lx1,lx2}}\label{newequrel2}
Two matrices $c\in M_n(R)$ and $d\in M_m(R)$ are said to be

$(1)$ \emph{$M$-equivalent} if there is a bijection $\pi:\mathcal{M}_c\ra\mathcal{M}_d$, such that $R[x]/(f(x)) \simeq R[x]/((f(x))\pi)$ as $R$-algebras and ${P_c(f(x))} = {P_d((f(x))\pi)}$ for all $f(x)\in \mathcal{M}_c$. In this case, we write $c\stackrel{M}\sim d$.

$(2)$  \emph{$D$-equivalent} if there is a bijection $\pi:\mathcal{M}_c\ra\mathcal{M}_d$, such that $R[x]/(f(x)) \simeq R[x]/((f(x))\pi)$ as $R$-algebras and $\mathcal{H}_{P_c(f(x))}= \mathcal{H}_{P_d((f(x))\pi)}$ for all $f(x)\in \mathcal{M}_c$. In this case, we write $c\stackrel{D}\sim d$.

$(3)$ \emph{$AD$-equivalent} if there is a bijection $\pi:\mathcal{M}_c\ra\mathcal{M}_d$, such that $R[x]/(f(x))\simeq R[x]/((f(x))\pi)$ as $R$-algebras and either $P_c(f(x))= {P_d((f(x))\pi)}$ or $P_c(f(x))=\mathcal{J}_{P_d((f(x))\pi)}$ for all $f(x)\in \mathcal{M}_c$. In this case, we write $c\stackrel{AD}\sim d$.

$(4)$ \emph{$S$-equivalent} if there is a bijection $\pi: \mathcal{R}_c\ra \mathcal{R}_d$, such that $R[x]/(f(x))\simeq R[x]/((f(x))\pi)$ as $R$-algebras and either $P_c(f(x))= P_d((f(x))\pi)$ or $P_c(f(x))=\mathcal{J}_{P_d((f(x))\pi)}$ for all $f(x)\in \mathcal{R}_c$. In this case, we write $c\stackrel{S}\sim d$.
\end{Def}

The interrelations among these equivalence relations are indicated in the following remark. In general, the converses of all these implications are not true.

\begin{Rem} \label{rber}\rm{
Let $c\in M_n(R)$ and $d\in M_m(R)$. Then the following implications hold.
$$\xymatrix@C=0.7cm@R=0.01cm{ & & & c\stackrel{D}\sim d \ar@{=>}[dr]&\\
    c\stackrel{I}\sim d \ar@{=>}[r] & c\stackrel{M}\sim d \ar@{=>}[r] & c\stackrel{AD}\sim d \ar@{=>}[ur]\ar@{=>}[dr] & & c\stackrel{Sg}\sim d.\\
    & & & c\stackrel{S}\sim d \ar@{=>}[ur] &
}$$

{\it Proof.} Clearly, $c\stackrel{I}\sim d\Ra c\stackrel{M}\sim d \Ra c\stackrel{AD}\sim d \Ra c\stackrel{D}\sim d$ (by Lemma \ref{no-1}), and  $c\stackrel{AD}\sim d\Ra c\stackrel{S}\sim d$. We have the following facts.

(i) If $f(x)$ and $g(x)$ are irreducible polynomials and $i,j\in\mathbb{Z}_{>0}$ such that $R[x]/(f(x)^i)\simeq R[x]/(g(x)^j)$ as algebras, then $i=j$ and $R[x]/(f(x))\simeq R[x]/(g(x))$ as algebras.

(ii) If $f(x)\in \mathcal{M}_c$, then $|P_c(f(x))|=\max\{i\in P_c(f(x))\}$ if and only if $\mathcal{H}_{P_c(f(x))}=\{\{1,\cdots,1\}\}$ contains exactly $\max\{i\in P_c(f(x))\}$ elements if and only if $\mathcal{D}_{P_c(f(x))}=\emptyset$.

(iii) If there is a bijection $\pi:\mathcal{I}_c\ra\mathcal{I}_d$ such that $R[x]/(f(x)) \simeq R[x]/((f(x))\pi)$ as algebras and $\mathcal{D}_{P_c(f(x))}=\mathcal{D}_{P_d((f(x))\pi)}$ for all $f(x)\in \mathcal{I}_c$, then $c\stackrel{Sg}\sim d$.

Suppose $c\stackrel{D}\sim d$. Then, by definition, there is a bijection $\pi:\mathcal{M}_c\ra\mathcal{M}_d$ such that $R[x]/(f(x)) \simeq R[x]/((f(x))\pi)$ as algebras and $\mathcal{H}_{P_c(f(x))}= \mathcal{H}_{P_d((f(x))\pi)}$ for all $f(x)\in \mathcal{M}_c$. By (ii), $|P_c(f(x))|=\max \{i\in P_c(f(x))\}$ if and only if $|P_d((f(x))\pi)|=\max\{j\in P_d((f(x))\pi)\}$. Hence $\pi$ induces a bijection $\pi^\prime: \mathcal{I}_c\ra\mathcal{I}_d$ such that $R[x]/(f(x)) \simeq R[x]/((f(x))\pi')$ as algebras and $\mathcal{H}_{P_c(f(x))}= \mathcal{H}_{P_d((f(x))\pi')}$ for all $f(x)\in \mathcal{I}_c$ by (i). Hence $\mathcal{D}_{P_c(f(x))}=\mathcal{D}_{P_d((f(x))\pi')}$ for all $f(x)\in \mathcal{I}_c$. It follows from (iii) that $c\stackrel{Sg}\sim d$.

Suppose $c\stackrel{S}\sim d$. Then, by definition, there is a bijection $\pi: \mathcal{R}_c\ra \mathcal{R}_d$ such that $R[x]/(f(x))\simeq R[x]/((f(x))\pi)$ as $R$-algebras and either $P_c(f(x))= P_d((f(x))\pi)$ or $P_c(f(x))=\mathcal{J}_{P_d((f(x))\pi)}$ for all $f(x)\in \mathcal{R}_c$. Hence $\mathcal{H}_{P_c(f(x))}= \mathcal{H}_{P_d((f(x))\pi)}$ for all $f(x)\in \mathcal{R}_c$. Furthermore, for an irreducible polynomial $f(x)\in \mathcal{M}_c$, we have $|P_c((f(x))\pi)|=\max\{j\in P_c((f(x))\pi)\}=1$. Similarly to the case of $c\stackrel{D}\sim d$, we can prove $c\stackrel{Sg}\sim d$. $\square$
}
\end{Rem}

In general, each inverse implications of the above diagram is not true.
Finally, we give examples to illustrate $I$-equivalences and $Sg$-equivalences. One of the examples shows in general that Sg-equivalences do not have to imply $D$-equivalence nor $S$-equivalences.
\begin{Bsp}\label{newe}{\rm
Let $R$ be a field and $J_n(\lambda)$ the $n\times n$ Jordan matrix with the eigenvalue $\lambda\in R$.

(1) Similar matrices are $I$-equivalent, but the converse is not true in general.  Let $c:=J_3(0)\oplus J_3(0)\in M_6(R)$ and $d:=J_3(1)\oplus J_3(1)\in M_6(R)$, where $\oplus$ stands for forming the diagonal block matrix. Then $\tilde{\mathcal{E}}_c=\{\{x^3,x^3\}\}, \tilde{\mathcal{E}}_d=\{\{(x-1)^3,(x-1)^3\}\}$, $\mathcal{M}_c=\{x^3\}, \mathcal{M}_d=\{(x-1)^3\}, \tilde{P}_c(x^3)=\{\{3,3\}\}=\tilde{P}_d((x-1)^3)$, $m_c(x)=x^3, m_d(x)=(x-1)^3$. If we define $\pi:\mathcal{M}_c\to \mathcal{M}_d, x^3\mapsto (x-1)^3$, then $c\stackrel{I}\sim d$. Since $c$ and $d$ have different minimal polynomials, they are not similar. Thus the $I$-equivalence is a proper generalization of the similarity relation of matrices.

(2) Let $c:=J_2(0)\oplus J_4(0)\in M_6(R)$ and $d:=J_2(0)\oplus J_2(1)\in M_4(R)$. Clearly, $\mathcal{E}_c=\{x^2,x^4\}, \mathcal{E}_d=\{x^2,(x-1)^2\}=\mathcal{M}_d=\mathcal{R}_d, \mathcal{M}_c=\{x^4\}=\mathcal{R}_c, \mathcal{I}_c=\{x\}, \mathcal{I}_d=\{x,x-1\}, r_c=1=r_d, Q_{c,1}\simeq R[x]/(x)\simeq Q_{d,1}$ and $\mathcal{D}_{c,1}=\{2,2\}=\mathcal{D}_{d,1}$. Thus $c$ and $d$ are $Sg$-equivalent. Now, it follows from $|\mathcal{M}_c|=|\mathcal{R}_c|=1\neq 2=|\mathcal{M}_d|=|\mathcal{R}_d|$ that $c$ and $d$ are neither $D$-equivalent nor $S$-equivalent.
}\end{Bsp}

\section{Singularity categories and singular equivalences}\label{sec4}
In this section we describe the singularity categories and singular equivalences of centralizer matrix algebras over arbitrary fields. To this purpose, we study singularity categories of endomorphism algebras of generators over quotients of polynomial algebras. In this way, we describe the singularity categories of centralizer matrix algebras as  products of stable module categories. Note that, when $R=\mathbb{C}$, the singularity categories and singular equivalences of endomorphism algebras of generators over quotients of polynomial algebras have been investigated in \cite[Remark 4.11, Corollary 4.13]{kk1}.
However, the argument there relies on geometric methods, in this case, algebraically closed field is reasonably assumed, whereas we provide a purely algebraic and elementary approach over arbitrary fields.

\subsection{Singularity categories}
In this subsection we give a precise description of the singularity categories of centralizer matrix algebras. But we start with a more general setting of exact categories.

Let $(\A,\mathcal{S})$ be an exact category, and let $\mathcal{C}$ be a full additive subcategories of $\A$.
Suppose that $\mathcal{S}'$ is a subclass of $\mathcal{S}$ such that it is closed under isomorphisms. The category $\mathcal{C}\subseteq \A$ is said to be \emph{closed under admissible push-outs of $\mathcal{S}'$} if, for any admissible monomorphism $f:X\to Y$ of $\mathcal{S}'$ and homomorphism $\phi:X\to Z$ in $\mathcal{C}$, the push-out of $(f,\phi)$ in $\mathcal{A}$ belongs to $\mathcal{C}$. Dually, the category $\mathcal{C}\subseteq \A$ is said to be \emph{closed under admissible pull-backs of $\mathcal{S}'$} if, for any admissible epimorphism $g:Y\to Z$ of $\mathcal{S}'$ and homomorphism $\psi:X\to Z$ in $\mathcal{C}$, the pull-back of $(g,\psi)$ in $\mathcal{A}$ belongs to $\mathcal{C}$.

\begin{Lem}\label{popoapb}
Let $(\A,\mathcal{S})$ be an exact category with a full additive subcategory $\mathcal{D}$, and let $X \xr{f} Y \xr{g} Z$ be an admissible exact sequence of $\mathcal{S}$ such that $f$ and $g$ are $\mathcal{D}$-monic and $\mathcal{D}$-epic, respectively. 

$(1)$ For any morphism $\phi:X\to X^\prime$ in $\A$, there is the push-out diagram in $\A$:

$$\xymatrix@R=0.5cm{
    X \ar[r]^{f}\ar[d]_{\phi} & Y \ar[r]^{g}\ar[d]^{\phi'} & Z \ar@{=}[d] \\
    X' \ar[r]^{f'} & Y' \ar[r]^{g'} & Z  ,
}$$
with $(f',g')$ an admissible exact sequences of $\mathcal{S}$ such that $f'$ and $g'$ are $\mathcal{D}$-monic and $\mathcal{D}$-epic, respectively.

$(2)$ Dually, for any morphism $\psi:Z'\to Z$ in $\A$, there is the pull-back diagram in $\A$:

$$\xymatrix@R=0.5cm{
     X \ar[r]^{f'}\ar@{=}[d] & Y' \ar[r]^{g'}\ar[d]_{\psi'} & Z' \ar[d]^{\psi} \\
     X \ar[r]^{f} & Y \ar[r]^{g} & Z  ,
}$$
with $(f',g')$ an admissible exact sequences of $\mathcal{S}$ such that $f'$ and $g'$ are $\mathcal{D}$-monic and $\mathcal{D}$-epic, respectively.
\end{Lem}

{\it Proof.} Since (2) is dual to (1), we only prove (1). The existence of the diagram follows from the axioms of exact categories. Thus, it remains only to show that $f'$ and $g'$ are $\mathcal{D}$-monic and $\mathcal{D}$-epic, respectively. Since $g=\phi'g'$ is $\mathcal{D}$-epic, $g'$ is $\mathcal{D}$-epic. It remains to show that $f'$ is $\mathcal{D}$-monic.

Let $\delta:X'\to D$ be a morphism in $\A$ with $D\in \mathcal{D}$. Since $f$ is $\mathcal{D}$-monic, there is a morphism $\rho:Y\to D$ such that $f\rho=\phi\delta$. By the universal property of push-outs, there is a morphism $\gamma:Y'\to D$ in $\A$ such that $\delta=f'\gamma$. This means that $f'$ is $\mathcal{D}$-monic. $\square$

The following lemma is straightforward and its proof is omitted.

\begin{Lem}\label{exactcat}
Let $\mathcal{C}$ and $\mathcal{D}$ be full additive subcategories of an exact category $(\A,\mathcal{S})$ with $\mathcal{D}\subseteq \mathcal{C}$. Define $$\mathcal{S}_\mathcal{D} := \left\{
C_1 \xrightarrow{f_1} C_2 \xrightarrow{f_2} C_3
\ \middle|\
\begin{array}{l}
(f_1,f_2)\ \text{is a kernel-cokernel pair in}\ \mathcal{S}, f_1 \text{ and } f_2 \text{ are } \\ \mathcal{D}\text{-monic} \text{ and } \mathcal{D}\text{-epic} \text{ in } \mathcal{C}, \text{respectively}
\end{array}
\right\}.$$ Suppose that $\mathcal{C}\subseteq \A$ is closed under admissible push-outs and pull-backs of $\mathcal{S}_\mathcal{D}$. Then

$(1)$ $(\mathcal{C},\mathcal{S}_\mathcal{D})$ is an exact category.

$(2)$ If $\mathcal{D}\subseteq \mathcal{C}$ is closed under direct summands and if, for all $C \in \mathcal{C}$, there exist an admissible monomorphism $f:C \to D$ and an admissible epimorphism $g:D^\prime \to C$ of $\mathcal{S}_\mathcal{D}$ with $D, D^\prime \in \mathcal{D}$, then $(\mathcal{C}, \mathcal{S}_\mathcal{D})$ is a Frobenius category with ${\rm Inj}_{\mathcal{S}_\mathcal{D}}(\mathcal{C})=\mathcal{D}$. Consequently, $\mathcal{C}/\mathcal{D}$ is a triangulated category.
\end{Lem}

Let $\Lambda$ be an Artin algebra and $M\in\Lambda\modcat$. 
The \emph{$M$-stable category} $\Lambda\modcat/[M]$ of $\Lambda\modcat$ is defined to be the quotient category of $\Lambda\modcat$ modulo $\add(M)$. For $X_1, X_2\in A\modcat$, we say that $X_1$ and $X_2$ are \emph{$M$-stably isomorphic} if  they are isomorphic in $\Lambda\modcat/[M]$, equivalently, there exist $M_1$ and $M_2$ in $\add(M)$ such that $X_1\oplus M_1\simeq X_2\oplus M_2$ as $\Lambda$-modules. For a full additive subcategory $\mathcal{C}$ of $\Lambda\modcat$, we denote by $\mathcal{C}/[M]$ the full subcategory of $\Lambda\modcat/[M]$ consisting of all objects $X$ which are $M$-stably isomorphic to objects in $\mathcal{C}$.

\textbf{From now on}, we fix an irreducible polynomial $f(x) \in R[x]$ of positive degree $u$, and set $A := R[x]/(f(x)^n)$ for a positive integer $n$. We keep the notation in Section \ref{sect2.3}.

\begin{Lem}\label{tcf}
Let $a,b\in \{0,1,\cdots,n\}$ with $a<b$, and let $M:=\bigoplus_{i=a}^b M(i)$ and $N:=M(a)\oplus M(b)$ be $A$-modules. Then

$(1)$ $(\add(M),\mathcal{S}_{\add(N)})$ is a Frobenius $R$-category with ${\rm Inj}_{\mathcal{S}_{\add(N)}}(\add(M))=\add(N)$.

$(2)$ $\add(M)/[N]\simeq \stmc{R[x]/(f(x)^{b-a})}$ as triangulated $R$-categories.
\end{Lem}

{\it Proof.} (1) Let $K:= \bigoplus_{i=0}^aM(i)\oplus\bigoplus_{j=b}^{n}M(j)$. 
We first show the following statement.

\smallskip
$(\dag)$ $\add(M)\subseteq A\modcat$ is closed under admissible push-outs and admissible pull-backs of $\mathcal{S}_{\add(N)}$.  

\smallskip
Actually, since $A$ is a local Nakayama algebra, the following facts hold for an $A$-module $X$:

1) $\#(X)=\#(\soc(X))$.

2) $X\in \add(M)$ if and only if $\ell_A(\soc^{a}(X))= a(\#(X))$ and $LL(X)\le b$.

Suppose that $f:X\to Y$ is an admissible monomorphism of $\mathcal{S}_{\add(N)}$, and that $\phi:X\to X^\prime$ is an arbitrary morphism in $\add(M)$. By Lemma \ref{popoapb}(1), we have the push-out diagram in $A\modcat$:
$$\xymatrix@R=0.5cm{
    0\ar[r] & X\ar[r]^f\ar[d]_{\phi} & Y\ar[r]^g\ar[d] & Z\ar@{=}[d] \ar[r] & 0\\
    0\ar[r] & X^\prime \ar[r]^{f'} & Y^\prime \ar[r]^{g'} & Z\ar[r] & 0.
}$$
where $f'$ and $g'$ are $\add(N)$-monic and $\add(N)$-epic, respectively. This diagram provides an exact sequence: $0\to X\xr{(f,\phi)} Y\oplus X^\prime \to Y^\prime \to 0$ in $A\modcat$, and shows $LL(Y^\prime)\le LL(Y\oplus X^\prime)\le LL(M)\le b$.

To prove $Y'\in \add(M)$, we only need to show $\ell_A(\soc^{a}(Y'))=a(\#(Y'))$ by the fact 2). The case $a\le 1$ is trivial, so we may assume $a>1$.
Since $Z\in \add(M)$ and $g'$ is $\add(N)$-epic, it follows from Corollary \ref{esam}(2) that $g'$ is $\add(K)$-epic.
Since $M(1)$ and $M(a)$ belong to $\add(K)$,  $g'$ is both $\add(M(1))$-epic and $\add(M(a))$-epic. By Lemma \ref{exies}, we have two exact sequences in $A\modcat$:
$$(\diamondsuit)\quad 0\lra \soc(X')\lra \soc(Y')\lra \soc(Z)\lra 0\,\text{ and }$$ $$(\diamondsuit\diamondsuit)\quad 0\lra \soc^a(X')\lra \soc^a(Y')\lra \soc^a(Z)\lra 0.$$  It then follows from 1) and $(\diamondsuit)$ that $\#(Y')=\#(\soc(Y'))=\#(\soc(X'))+\#(\soc(Z))=\#(X')+\#(Z)$. Since both $X'$ and $Z$ lie in $\add(M)$, we infer from $(\diamondsuit\diamondsuit)$ and 2) that $\ell_A(\soc^a(Y'))=\ell_A(\soc^a(X'))+\ell_A(\soc^a(Z))=a\big(\#(X')+\#(Z)\big)$, and therefore $\ell_A(\soc^a(Y'))=a(\#(X')+\#(Z))=a(\#(Y'))$. By the fact 2), $Y'\in \add(M)$. Hence $\add(M)\subseteq A\modcat$ is closed under admissible push-outs of $\mathcal{S}_{\add(N)}$.
Similarly, we can prove that $\add(M)\subseteq A\modcat$ is closed under admissible pull-backs of $\mathcal{S}_{\add(N)}$. This shows $(\dag)$.

Now, it follows from Corollary \ref{esam}(1) and Lemma \ref{exactcat} that $(\add(M),\mathcal{S}_{\add(N)})$ is a Frobenius $R$-category with ${\rm Inj}_{\mathcal{S}_{\add(N)}}(\add(M))=\add(N)$.

(2) Any homomorphism $f:X\to Y$ in $A\modcat$ restricts to a homomorphism $f^\prime:\soc^a(X)\to \soc^a(Y)$. Thus there exists a unique homomorphism $f^\wedge :X/\soc^a(X)\to Y/\soc^a(Y)$ of $A$-modules, fitting into the exact commutative diagram:
$$(\star)\quad \xymatrix@R=0.5cm{
  0\ar[r]& \soc^a(X)\ar[r]\ar[d]_-{f'} &X\ar[r]^-{\pi_X}\ar[d]_f & X/\soc^a(X)\ar[d]_-{f^\wedge}\ar[r] &0\\
 0\ar[r]& \soc^a(Y)\ar[r] &  Y\ar[r]^-{\pi_Y} & Y/\soc^a(Y) \ar[r] &0,
}$$
where $\pi_X$ and $\pi_Y$ are the canonical surjective homomorphisms. Hence we can define a functor
   $$ F:\; \add(M)\lra R[x]/(f(x)^{b-a})\modcat, $$
   $$ X\mapsto X/\soc^a(X) \ \text{ for }\ X\in \add(M),\quad
    f\mapsto f^\wedge \ \text{ for }\ f:X\to Y \ \text{ in } \add(M).$$
It is easy to see that $F$ is a well-defined $R$-functor and $F(\add(N))=\add(R[x]/(f(x)^{b-a}))$.

Let $B:=R[x]/(f(x)^{b-a})$. Since $\add(N)$-epic morphisms are $\add(M(a))$-epic morphisms, the functor $F$ sends admissible exact sequences in $\mathcal{S}_{\add(N)}$ to exact sequences in $B\modcat$ by Lemma \ref{exies}. Moreover, $F$ induces a triangle $R$-functor $\underline{F}:\add(M)/\add(N)\to \stmc{B}$. Now, we show that $\underline{F}$ is an equivalence.

Clearly, $\underline{F}$ is dense.
For $X, Y\in\add(M)\subseteq A\modcat$, we denote by $\StHom_N(X,Y)$ the Hom-set of $X$ and $Y$ in $\add(M)/\add(N)$, and will show that $\underline{F}:\StHom_N(X,Y)\to \StHom_B(X,Y)$ is an isomorphism of $R$-modules.

We may suppose that $X$ and $Y$ are indecomposable. If $X\in \add(N)$ or $Y\in \add(N)$, then  $\underline{F}$ is clearly an isomorphism. Now suppose $X=M(i)$ and $Y=M(j)$ with $a<i,j <b$. Then $F(X)=M(i-a), F(Y)=M(j-a)$ and  $gg_{j,j-a}=g_{i,i-a}g^{\wedge}=g_{i,i-a}F(g)$ for $g: X\to Y$ in $\add(M)$ (see the diagram $(\star)$). For $\tilde{f}:F(X)\to F(Y)$ in $B\modcat$, Lemma \ref{ftapi}(2) guarantees a morphism $f: X\to Y \in\add(M)$ such that $\underline{F}(\underline{f})=\underline{\tilde{f}}$. This shows that $\underline{F}$ is full.

To see that $\underline{F}$ is faithful, we pick up a homomorphism $g:X\to Y$ in $\add(M)$ such that $\underline{F}(\underline{g})=0$ in $\stmc{B}$, that is, $\underline{F(g)}=0$ in $\stmc{B}$. This means that $F(g)$ factorizes through a projective $B$-module, and therefore through a projective cover $\beta:B\to F(Y)$ of $F(Y)$. Hence there exists a homomorphism $\alpha:F(X)\to B$ of $B$-modules such that $F(g)=\alpha\beta$. By Lemma \ref{ftapi}(2), there exist two homomorphisms $\tilde{\alpha}:X\to M(b)$ and $\tilde{\beta}:M(b)\to Y$ of $A$-modules such that $\tilde{\alpha}g_{b,b-a}=g_{i,i-a}\alpha$ and $\tilde{\beta}g_{j,j-a}=g_{b,b-a}\beta$. Thus  we get the following commutative diagram in $A$-mod:

$$\xymatrix@R=0.5cm{ X\ar[r]^-{\tilde{\alpha}}\ar[d]^-{g_{i,i-a}} &M(b)\ar[r]^-{\tilde{\beta}}\ar[d]^-{g_{b,b-a}} & Y\ar[d]^-{g_{j,j-a}}\\
F(X)\ar[r]^-{\alpha} &B\ar[r]^-{\beta}& F(Y).\\
} $$
Here, $B$-mod is regarded as a full subcategory of $A$-mod. Due to $gg_{j,j-a}=g_{i,i-a}F(g)$, we get $(g-\tilde{\alpha}\tilde{\beta})g_{j,j-a}=gg_{j,j-a}-\tilde{\alpha}\tilde{\beta}g_{j,j-a}=g_{i,i-a}F(g)-\tilde{\alpha}g_{b,b-a}\beta
=g_{i,i-a}F(g)-g_{i,i-a}\alpha\beta=0$. This shows that $g-\tilde{\alpha}\tilde{\beta}$ factorizes through $\Ker(g_{j,j-a})\simeq M(a)$. Consequently, there exist two homomorphisms $h_1:X\to M(a)$ and $h_2:M(a)\to Y$ in $A$-mod such that $g-\tilde{\alpha}\tilde{\beta}=h_1h_2$. Let $\gamma:=(h_1,\tilde{\alpha}):X\to M(a)\oplus M(b)$ and $\delta:=\left(\begin{smallmatrix} h_2 \\ \tilde{\beta}   \end{smallmatrix}\right):M(a)\oplus M(b)\to Y$. Then $g=\gamma\delta$, that is, $\underline{g}=0$ in $\add(M)/\add(N)$. Hence $F$ is faithful. $\square$

\medskip
The next lemma characterizes the singularity categories of the endomorphism algebras of a generators over quotient algebras of polynomial algebras.

\begin{Prop}\label{scoeog}
Let $T=\{t_1,\cdots,t_s\}$ be a subset of $[n]$ with $t_1>t_2>\cdots>t_s$, and set $M:=\bigoplus_{i=1}^sM(t_i)$, $\Gamma:=\End_A(M)$. Then there is a triangle equivalence of the triangulated $R$-categories:
$$\Ds{\Gamma}\simeq \prod_{i\in\mathcal{D}_T}\stmc{R[x]/(f(x)^{i})}.$$
\end{Prop}

{\it Proof.} Without loss of generality, we may assume $t_1=n$. Let $\A=\mathcal{C}=A\modcat$ and $\mathcal{D}=\add(M)$. It follows from Corollary \ref{esam} and Lemma \ref{exactcat}
that $(A\modcat, \mathcal{S}_{\add(M)})$ is a Frobenius $R$-category with ${\rm{Inj}}_{\mathcal{S}_{\add(M)}}(A\modcat)= \add(M)$.
Since $\{M(i)\mid i\in [n]\}$ is a complete set of pairwise non-isomorphic indecomposable $A$-modules, we have $\gd(\End_A(\oplus_{i=1}^nM(i)))\le 2$. It follows from \cite[Theorem 1.1(2)]{kiwy1} that the equivalence $\Hom_A(M,-):A\modcat\to \gp{\Gamma}$ of Frobenius $R$-categories induces a triangle equivalence $A\modcat/[M]\to \Stgp{\Gamma}\simeq \Ds{\Gamma}$.

In the following, we describe explicitly the category $A\modcat/[M]$ by two steps (1)-(2).

(1) For $i\in [s]$, $M_{i+1,i}:=\bigoplus_{j=t_{i+1}}^{t_{i}}M(j)$ and $N_{i+1,i}:=M(t_{i+1})\oplus M(t_i)$, we prove that the pair $(\add(M_{i+1,i}),$ $\mathcal{S}_{\add(N_{i+1,i})})$ is an exact $R$-subcategory of $(A\modcat, \mathcal{S}_{\add(M)})$. Here, for convenience, we understand $t_{s+1}=0$.

Indeed, since $M\in \add(\bigoplus_{j=0}^{t_{i+1}}M(j)\oplus\bigoplus_{k=t_i}^{n}M(k))$, $\add(N_{i+1,i})$-monic (or epic) morphisms in $\add(M_{i+1,i})$ are $\add(M)$-monic (or epic) in $A\modcat$ by Corollary \ref{esam}(2). Hence $\mathcal{S}_{\add(N_{i+1,i})}=\{(f,g)\in\mathcal{S}_{\add(M)}\mid f, g\in \add(M_{i+1,i})\}$.
It remains to show that $(\add(M_{i+1,i}),\mathcal{S}_{\add(N_{i+1,i})})$ is closed under extensions in $(A\modcat, \mathcal{S}_{\add(M)})$. Precisely, given an admissible exact sequence $X\xr{f}Y\xr{g}Z$ of $\mathcal{S}_{\add(M)}$ with $X, Z\in \add(M_{i+1,i})$, we have to show $Y\in \add(M_{i+1,i})$. In fact, by Corollary \ref{esam}(1), we have an admissible exact sequence of $\mathcal{S}_{\add(N_{i+1,i})}$: $0\to Z^\prime \lraf{h} M' \to Z\to 0$ with $M'\in \add(N_{i+1,i})$ and $Z'\in \add(M_{i+1,i})$. Since $g$ is $\add(N_{i+1,i})$-epic, we can get the exact commutative diagram in $A\modcat$:
$$\xymatrix@R=0.7cm{
    0\ar[r] & Z' \ar@{-->}[d]_{\phi} \ar[r]^-{h} & M' \ar[r] \ar@{-->}[d] & Z \ar[r]\ar@{=}[d] & 0\\
    0\ar[r] & X \ar[r]^{f} & Y \ar[r]^{g} & Z \ar[r] & 0.
}$$
This diagram is a push-out diagram in $A\modcat$. Since $h$ is an admissible monomorphism of $\mathcal{S}_{\add(N_{i+1,i})}$ and $\phi\in\add(M_{i+1,i})$, we deduce from the statement $(\dag)$ in the proof of Lemma \ref{tcf} that $Y\in \add(M_{i+1,i})$. This implies that $(\add(M_{i+1,i}),\mathcal{S}_{\add(N_{i+1,i})})$ is closed under extensions in $(A\modcat,\mathcal{S}_{\add(M)})$.

(2) We show $A\modcat/[M]\simeq \prod_{i\in\mathcal{D}_T}\stmc{R[x]/(f(x)^{i})}$ as triangulated $R$-categories. As a consequence, we have the triangle equivalence of triangulated $R$-categories: $\Ds{\Gamma}\simeq \prod_{i\in\mathcal{D}_T}\stmc{R[x]/(f(x)^{i})}.$

Actually, Corollary \ref{eeite} and Lemma \ref{tcf}(1) show that the embedding functor $\lambda_{i+1,i}$ induces a fully faithful triangle $R$-functor
$$\underline{\lambda_{_{i+1,i}}}:\add(M_{i+1,i})/\add(N_{i+1,i})\lra A\modcat/[M].$$ Thus we obtain the full triangulated $R$-subcategories
$ \Img(\underline{\lambda_{2,1}}), \Img(\underline{\lambda_{3,2}}),\cdots, \Img(\underline{\lambda_{s+1,s}})$ of $A\modcat/[M]$.
By Lemma \ref{ftapi}(1), there are no nonzero morphisms between such two distinct subcategories. Trivially, we have $A\modcat=\add(M_{2,1}\oplus \cdots \oplus M_{s+1,s})$. Now, it follows from Lemma \ref{acstb} that there is a triangle equivalence $A\modcat/[M]\simeq\prod_{i=1}^{s} \;\; \add(M_{i+1,i})/\add(N_{i+1,i})$ of triangulated $R$-categories. By Lemma \ref{tcf}(2),  $\add(M_{i+1,i})/\add(N_{i+1,i})\simeq \stmc{R[x]/(f(x)^{t_i-t_{i+1}})}$ and $A\modcat/[M]\simeq \prod_{i=1}^{s}\stmc{R[x]/(f(x)^{t_i-t_{i+1}})}$ as triangulated $R$-categories. Due to $\stmc{R[x]/(f(x))}=0$, we rewrite the latter equivalence as $A\modcat/[M]\simeq \prod_{i\in\mathcal{D}_T}\stmc{R[x]/(f(x)^{i})}$ as triangulated $R$-categories. $\square$

\begin{Koro}\label{gdoeog}
Let $T=\{t_1,\cdots,t_s\}$ be a subset of $[n]$ with $t_1>t_2>\cdots>t_s$, $M:=\bigoplus_{i=1}^s(M(t_i))$ and $\Gamma:=\End_A(M)$. Then
$$\gd(\Gamma)=\begin{cases} 0, & \text { if } t_1=1, \\ 2, &\text { if } t_1\neq 1 \text{ and } M \text{ is an additive generator for } R[x]/(f(x)^{t_1})\modcat, \\ \infty, &\text { otherwise.}   \end{cases}$$
In particular, $\gd(\Gamma)<\infty$ if and only if $\Gamma$ is the Auslander algebra of $R[x]/(f(x)^{t_1})$.
\end{Koro}

{\it Proof.} For an Artin algebra $\Lambda$, $\gd(\Lambda)<\infty$ if and only if $\Ds{\Lambda}=0$. By Proposition \ref{scoeog}, we know $\Ds{\Gamma}\simeq \prod_{i\in\mathcal{D}_T}\stmc{R[x]/(f(x)^{i})}$ as triangulated categories. Thus $\Ds{\Gamma}=0$ if and only if $\mathcal{D}_T=\emptyset$ if and only if $T=[t_1]$. Equivalently, $\Ds{\Gamma}=0$ if and only if $M$ is an additive generator for $R[x]/(f(x)^{t_1})\modcat$. In this case, $\Gamma$ is a semisimple algebra if and only if $R[x]/(f(x)^{t_1})$ is a semisimple algebra if and only if $t_1=1$. $\square$

\smallskip
Remark that if $c$ has a Jordan normal form then Corollary \ref{gdoeog} follows from the cellularity of $S_n(c,R)$ by \cite{xz-LAA}. For an arbitrary field $R$, however, there may not exist any Jordan normal forms of matrices.

\smallskip
As a consequence of Proposition \ref{scoeog}, we give a description of singularity categories of centralizer matrix algebras. We first recall some notation and basic facts.

For $c\in M_n(R)$, let $m_c(x)$ be the minimal polynomial of $c$ over $R$ and $A_c:=R[x]/(m_c(x))$. We write
$$m_c(x):=\prod^{l_c}_{i=1} f_i(x)^{n_i} \mbox{  for } n_i\ge 1 \; \mbox{ and } \;U_i:=R[x]/(f_i(x)^{n_i})\mbox{  for }  i\in [l_c]$$ where $f_1(x),\cdots,f_{l_c}(x)$ are distinct irreducible (monic) polynomials in $R[x]$. Then $A_c\simeq U_1\times U_2\times\cdots\times U_{l_c}$ as $R$-algebras.

Since $A_c\simeq R[c]$ and the $R[c]$-module $R^n=\{(a_1,a_2,\cdots,a_n)^{tr}\mid a_i\in R, i\in[n]\}$ can be regarded as an $A_c$-module, we can decompose the $A_c$-module $R^n$, according to the blocks of $A_c$, in the following way:
$$(\star)\quad \quad R^n\simeq \bigoplus^{l_c}_{i=1}\bigoplus^{s_i}_{j=1} \; R[x]/(f_i(x)^{e_{ij}}),$$
where $s_i$ and $e_{ij}$ are positive integers. Note that $\{f_i(x)^{n_i}\mid i\in[l_c]\}=\mathcal{M}_c$ and  $\{\{f_i(x)^{e_{ij}}\mid i\in [l_c],  j\in [s_i]\}\}=\tilde{\mathcal{E}}$(see \cite[Chapter 4]{c1}, where $(\star)$ is stated in terms of invariant subspaces of a linear transformation).

Let $M_i:= \bigoplus_{j\in \tilde{P}_c(f_i(x)^{n_i})} R[x]/(f_i(x)^j)$ and $A_i:=\End_{U_i}(M_i)$ for $i\in  [l_c]$. Then all algebras $A_i$ are indecomposable and
$$S_n(c,R)\simeq \End_{R[c]}(R^n)\simeq \End_{A_c}(\bigoplus_{i=1}^{l_c}M_i)=\prod^{l_c}_{i=1}{\End}_{U_i}(M_i)= \prod^{l_c}_{i=1}A_i .$$
This is a decomposition of blocks of $S_n(c,R)$.
The following lemma is an immediate consequence of $(\star)$.

\begin{Lem} \label{bijection}{\rm \cite[Lemma 2.16]{lx1}}
There is a bijection $\pi$ from $\mathcal{E}_c$ to the set of non-isomorphic, indecomposable direct summands of the $A_c$-module $R^n$, sending $h(x)\in \mathcal{E}_c$ to the $A_c$-module $R[x]/(h(x))$.
\end{Lem}

The singular categories of centralizer matrix algebras can be characterized as follows.

\begin{Theo}\label{scocma}
For $c\in M_n(R)$, we have the triangle equivalence of triangulated $R$-categories:
$$\Ds{S_n(c,R)}\simeq \prod_{f(x)\in \mathcal{I}_c}\prod_{j\in \mathcal{D}_{P_c(f(x))}}\stmc{R[x]/(f(x)^j)}.$$
\end{Theo}

{\it Proof.} Clearly, for $i\in [l_c]$, we have $P_c(f_i(x))=\{e_{ij}\mid j\in [s_i]\}$ and $\mathcal{B}(M_i)\simeq \bigoplus_{r\in {P_c(f_i(x))}} R[x]/(f_i(x)^r)$ as $U_i$-modules. Since $\End_{U_i}(M_i)$ and $\End_{U_i}(\mathcal{B}(M_i))$ are Morita equivalent, $\Ds{\End_{U_i}(M_i)}\simeq \Ds{\End_{U_i}(\mathcal{B}(M_i))}$ as triangulated $R$-categories. By Proposition \ref{scoeog}, we know
$$\Ds{\End_{U_i}(M_i)}\simeq \Ds{\End_{U_i}(\mathcal{B}(M_i))}\simeq \prod_{j\in \mathcal{D}_{P_c(f_i(x))}}\stmc{R[x]/(f_i(x)^j)}$$ as triangulated $R$-categories. Thus $\Ds{S_n(c,R)}\simeq \prod_{i=1}^{l_c}\prod_{j\in \mathcal{D}_{P_c(f_i(x))}}\stmc{R[x]/(f_i(x)^j)}$ as triangulated $R$-categories. If  $|P_c(f_i(x))|=\max \{j\in P_c(f_i(x))\}$ for $i\in[l_c]$, then $\mathcal{D}_{P_c(f_i(x))}=\emptyset$. So we get the  triangle equivalence of triangulated $R$-categories:
$$\prod_{i=1}^{l_c}\prod_{j\in \mathcal{D}_{P_c(f_i(x))}}\stmc{R[x]/(f_i(x)^j)}\simeq \prod_{f(x)\in \mathcal{I}_c}\prod_{j\in \mathcal{D}_{P_c(f(x))}}\stmc{R[x]/(f(x)^j)}. \quad \square$$

\subsection{Singular equivalences }\label{sec5}
In this subsection, we study equivalences of the singularity categories of centralizer matrix algebras.

Firstly, we recall the Auslander-Reiten quiver of stable module category of an Artin algebra $\Lambda$. Let $\Gamma_{\Lambda}$ denote the Auslander-Reiten quiver of $\Lambda$, and let $\Gamma^s_{\Lambda}$ be the subquiver  of $\Gamma_{\Lambda}$ obtained from $\Gamma_{\Lambda}$ by removing all vertices corresponding to indecomposable projective modules and all arrows starting with or ending at projective vertices. For the local, symmetric Nakayama algebra $A=R[x]/(f(x)^n)$ with $f(x)$ an irreducible polynomial and $n\ge 2$ an integer, the quiver $\Gamma^s_{A}$ of $A$ is a connected quiver with $n-1$ vertices.

We note the following immediate consequence of \cite[Lemma X.1.2(d), p.336]{ars}.

\begin{Lem}\label{preserve} Suppose that $G$ is a stable equivalence between Artin algebras $\prod^{s}_{i=1}C_i$ and $\prod^{t}_{j=1}D_j$, where $C_i$ and $D_j$ are indecomposable non-semisimple algebras for $i\in [s]$ and $j\in [t]$. Suppose that the quivers $\Gamma^s_{C_i}$ and $\Gamma^s_{D_j}$ are connected for all $i\in [s]$ and all $j\in [t]$. Then $G$ preserves non-semisimple blocks, and therefore $s=t$.
\end{Lem}

Singular equivalences of centralizer matrix algebras can be characterized as follows.

\begin{Theo}\label{seocma}
Let $R$ be a field, $c\in M_n(R)$ and $d\in M_m(R)$. Then the following are equivalent.

$(i)$ $\Ds{S_n(c,R)}\simeq \Ds{S_m(d,R)}$ as triangulated $R$-categories.

$(ii)$ $\Ds{S_n(c,R)}\simeq \Ds{S_m(d,R)}$ as $R$-categories.

$(iii)$ $c$ and $d$ are $Sg$-equivalent.
\end{Theo}

{\it Proof.} By Theorem \ref{scocma}, we have the equivalences of triangulated $R$-categories:
$$\Ds{S_n(c,R)}\simeq \prod_{f(x)\in \mathcal{I}_c}\prod_{i\in \mathcal{D}_{P_c(f(x))}}\stmc{R[x]/(f(x)^i)}\ \text{ and }\  \Ds{S_m(d,R)}\simeq \prod_{g(x)\in \mathcal{I}_d}\prod_{j\in \mathcal{D}_{P_d(g(x))}}\stmc{R[x]/(g(x)^j)}.$$

(i) $\Ra$ (ii) This is clear since triangulated $R$-categories are $R$-categories.

(ii) $\Ra$ (iii) Suppose $\Ds{S_n(c,R)}\simeq \Ds{S_m(d,R)}$ as $R$-categories. Then the algebras
$$\prod_{f(x)\in \mathcal{I}_c}\prod_{i\in \mathcal{D}_{P_c(f(x))}}R[x]/(f(x)^i) \text{ and } \prod_{g(x)\in \mathcal{I}_d}\prod_{j\in \mathcal{D}_{P_d(g(x))}}R[x]/(g(x)^j)$$ are stably equivalent. Suppose $$F:\stmc{\prod_{f(x)\in \mathcal{I}_c}\prod_{i\in \mathcal{D}_{P_c(f(x))}}R[x]/(f(x)^i)}\lra \stmc{\prod_{g(x)\in \mathcal{I}_d}\prod_{j\in \mathcal{D}_{P_d(g(x))}}R[x]/(g(x)^j)}$$ is a stable equivalence. Since $R[x]/(h(x)^i)$ is not a semisimple algebra for any irreducible polynomial $h(x)\in R[x]$ and integer $i\ge 2$, it follows from Lemma \ref{preserve} that $F$ preserves non-semisimple blocks, that is, for $f(x)\in \mathcal{I}_c$ and $i\in \mathcal{D}_{P_c(f(x))}$, there exist an unique $g_{_f}(x)\in \mathcal{I}_d$ and $j_i\in \mathcal{D}_{P_d(g_{_f}(x))}$ such that $\stmc{R[x]/(f(x)^i)}$ and $\stmc{R[x]/(g_{_f}(x)^{j_i})}$ are equivalent as $R$-categories. This implies that $R[x]/(f(x))\simeq R[x]/(g_{_f}(x))$ as $R$-algebras and $i=j_i$ by Lemma \ref{St-i}. Thus $r_c=r_d$ and there is a permutation $\sigma\in \Sigma_{r_c}$ such that $Q_{c,i}\simeq Q_{d,(i)\sigma}$ as $R$-algebras and $\mathcal{D}_{c,i}=\mathcal{D}_{d,(i)\sigma}$. By definition, $c$ and $d$ are $Sg$-equivalent.

(iii) $\Ra$ (i) Suppose that $c$ and $d$ are $Sg$-equivalent, that is, $r_c=r_d$ and there is a permutation $\sigma\in \Sigma_{r_c}$ such that $Q_{c,i}:=R[x]/(p_{c,i}(x))\simeq Q_{d,(i)\sigma}:=R[x]/(p_{d,(i)\sigma}(x))$ as $R$-algebras and $\mathcal{D}_{c,i}=\mathcal{D}_{d,(i)\sigma}$ for $i\in [r_c]$ (see Section \ref{sect3} for notation). For $i\in [r_c]$, we deduce from  Lemma \ref{St-i} that the following triangle equivalences of triangulated $R$-categories hold:
$$\stmc{R[x]/(f(x)^s)}\simeq\stmc{R[x]/(p_{c,i}(x)^s)} \simeq \stmc{R[x]/(p_{d,(i)\sigma}(x)^s)}\simeq \stmc{R[x]/(g(x)^s)}$$ for $f(x)\in \mathcal{I}_{c,i}, g(x)\in \mathcal{I}_{d,(i)\sigma}$ and $s\in \mathbb{N}$. It follows from $\mathcal{D}_{c,i}=\mathcal{D}_{d,(i)\sigma}$ for $i\in [r_c]$ that
\begin{align*}
\Ds{S_n(c,R)}&\simeq \prod_{f(x)\in \mathcal{I}_c}\prod_{i\in \mathcal{D}_{P_c(f(x))}}\stmc{R[x]/(f(x)^i)}\simeq \prod_{i\in[r_c]}\prod_{j \in \mathcal{D}_{c,i}}\stmc{R[x]/(p_{c,i}(x)^j)}\\
&\simeq \prod_{i\in[r_c]}\prod_{j \in \mathcal{D}_{d,(i)\sigma}}\stmc{R[x]/(p_{d,(i)\sigma}(x)^j)},\ \text{ and }
\end{align*}
$$\Ds{S_m(d,R)}\simeq \prod_{g(x)\in \mathcal{I}_d}\prod_{j\in \mathcal{D}_{P_d(g(x))}}\stmc{R[x]/(g(x)^j)}\simeq \prod_{i\in[r_c]}\prod_{j\in \mathcal{D}_{d,i}}\stmc{R[x]/(p_{d,i}(x)^j)}$$
as triangulated $R$-categories. Hence $\Ds{S_n(c,R)}\simeq \Ds{S_m(d,R)}$ as triangulated $R$-categories. $\square$

\medskip
As a strong version of singular equivalences of algebras, we point out the following result about when centralizer matrix algebras are isomorphic.

\begin{Prop}  Let $R$ be a field, $c\in M_n(R)$ and $d\in M_m(R)$. Then

$S_n(c,R)\simeq S_m(d,R)$ as $R$-algebras if and only if $c$ and $d$ are $I$-equivalent.
\end{Prop}

\medskip
Though Morita, derived and stable equivalences preserve the number of non-semisimple blocks of centralizer matrix algebras, singular equivalences may not, in general,  preserve the number of non-semisimple blocks of centralizer matrix algebras.

\begin{Bsp}\label{senpnob}{\rm
Let $R$ be a field and $J_n(\lambda)$ the $n\times n$ Jordan matrix with the eigenvalue $\lambda\in R$. We take $c=J_2(0)\oplus J_4(0)\in M_6(R)$ and $d=J_2(0)\oplus J_2(1)\in M_4(R)$. Then $m_c(x)=x^4, m_d(x)=x^2(x-1)^2$, $\mathcal{E}_c=\{x^2,x^4\}, \mathcal{E}_d=\{x^2,(x-1)^2\}=\mathcal{M}_d=\mathcal{R}_d, \mathcal{M}_c=\{x^4\}=\mathcal{R}_c, \mathcal{I}_c=\{x\}, \mathcal{I}_d=\{x,x-1\}, r_c=1=r_d, Q_{c,1}\simeq R[x]/(x)\simeq Q_{d,1}$ and $\mathcal{D}_{c,1}=\{2,2\}=\mathcal{D}_{d,1}$. Thus $c$ and $d$ are $Sg$-equivalent, but not $D$-equivalent nor $S$-equivalent. By Theorem \ref{seocma}, $S_6(c,R)$ and $S_4(d,R)$ are singularly equivalent, but they are neither derived equivalent by \cite[Theorem 1.1]{lx1}, nor stably equivalent by \cite[Theorem 1.1]{lx2}.

In fact, we have $S_6(c,R)\simeq \End_{R[x]/(x^4)}(R[x]/(x^4)\oplus R[x]/(x^2))$ is indecomposable, while $S_4(d,R)\simeq R[x]/(x^2)\times R[x]/((x-1)^2)$ has two blocks of infinite global dimension. This shows that singularly equivalent centralizer matrix algebras may have different numbers of non-semisimple blocks in general, but stable equivalent centralizer matrix algebras have the same number of non-semisimple blocks (see \cite[Proposition 4.8]{lx2}).
}
\end{Bsp}

We may have a concise method to decide singular equivalences of some centralizer matrix algebras.
\begin{Bsp}\label{em-degree1}{\rm
Let $R$ be a field and $\lambda, \mu\in R$. Suppose that $H:=\{m_i\mid 1\le i\le s\}$ with $m_1>m_2>\cdots>m_s$ and $T:=\{n_j\mid 1\le j\le t\}$ with $n_1>n_2>\cdots>n_t$ are two sets of positive integers. We consider $A:=R[x]/((x-\lambda)^{m_1}), \Lambda:=\End_A(\bigoplus_{1\le i\le s} R[x]/((x-\lambda)^{m_i})$; $B:=R[x]/((x-\mu)^{n_1})$, and $\Gamma:=\End_B(\bigoplus_{1\le j\le t} R[x]/((x-\mu)^{n_j})$. By Theorem \ref{seocma}, $\Lambda$ and $\Gamma$ are singularly equivalent if and only if $\mathcal{D}_H=\mathcal{D}_T$.
}
\end{Bsp}

\section{Singular equivalences for permutation matrices}\label{sec6}
In this section, we study singular equivalences between centralizer matrix algebras of permutation matrices, and prove Theorem \ref{thm2}.

Throughout this subsection, $R$ denotes a field of characteristic $p \ge 0$.

Let $\sigma = \sigma_1 \cdots \sigma_s \in \Sigma_n$ and $\tau = \tau_1 \cdots \tau_t \in \Sigma_m$ be products of disjoint cycles, with cycle types $\lambda = (\lambda_1, \ldots, \lambda_s)$ and $\mu = (\mu_1, \ldots, \mu_t)$, respectively, where $\lambda_i \ge 1$ and $\mu_j \ge 1$ for all $i \in [s], j\in [t]$.

If $p>0$,  then a cycle $\sigma_i$ is said to be \emph{$p$-regular} if $p\nmid \lambda_i$, and \emph{$p$-singular} if $p\mid \lambda_i$. If $p=0$, all cycles are regarded as $p$-regular cycles. Let $r(\sigma)$ (respectively, $s(\sigma)$) be the product of the $p$-regular (respectively, $p$-singular) cycles of $\sigma$. We consider both $r(\sigma)$ and $s(\sigma)$ as elements in $\Sigma_n$. The permutation $\sigma$ is said to be \emph{$p$-regular} (or \emph{$p$-singular}) if  $\sigma = r(\sigma)$ (or $\sigma=s(\sigma)$).
For $p=0$, we have $r(\sigma)=\sigma$ and $s(\sigma)=id$, the identity permutation.

Let $c_{\sigma}:=\sum_{i=1}^ne_{i,(i)\sigma}\in M_n(R)$ denote the \emph{permutation matrix} of $\sigma$ over $R$, where $e_{ij}$ is the matrix with $1$ in $(i,j)$-entry and $0$ in all other entries.

\medskip
Our main result about the singularity categories of permutation matrices reads as follows.

\begin{Theo}\label{seispop}
Suppose $p\neq 2$ or $p=2$ and $\nu_p(\lambda_i)\neq 1\neq \nu_p(\mu_{j})$ for all $i\in [s]$ and $j\in[t]$. If $S_n(c_{\sigma},R)$ and $S_m(c_{\tau},R)$ are singularly equivalent, then so are $S_n(c_{s(\sigma)},R)$ and $S_m(c_{s(\tau)},R)$.
\end{Theo}

To prove Theorem \ref{seispop}, we first establish the following lemmas on elementary divisors and power index sets of permutation matrices.

\begin{Lem}{\rm \cite[Lemma 2.17]{lx1}}\label{per}
$\mathcal{E}_{c_{\sigma}} = \{g(x)^{p^{\nu_p(\lambda_i)}}\mid i\in[s], \; g(x)~\mbox{is an irreducible factor of}~ x^{\lambda_i}-1 \}.$
\end{Lem}

By Lemma \ref{per}, the matrix $c_\sigma$ always has a unique maximal elementary divisor of the form $(x-1)^{p^i}$ for some  $i\in \mathbb{N}$. This elementary divisor is called the \emph{exceptional} elementary divisor of $c_\sigma$.

\begin{Lem}\label{sseoiff}
If $p\neq 2$ or both $p=2$ and $P_{c_\sigma}(f(x))\neq \{1,2\}$ for all $f(x)\in \mathcal{M}_{c_\sigma}$, then

$(1)$ $s(\sigma)=id$ if and only if $r(\sigma)=\sigma$ if and only if $\mathcal{I}_{c_\sigma}=\emptyset$.

$(2)$ $\mathcal{I}_{c_{s(\sigma)}}=\mathcal{I}_{c_\sigma}$, $r_{c_{s(\sigma)}}= r_{c_{\sigma}}$ and, for $f(x)\in \mathcal{I}_{c_\sigma}$,  $$P_{c_{s(\sigma)}}(f(x))=\begin{cases}P_{c_\sigma}(f(x)), & \mbox{ if } f(x)=x-1,\\ P_{c_\sigma}(f(x))\setminus\{1\}, & \text{ otherwise}.\end{cases}$$
\end{Lem}

{\it Proof.} Without loss of generality, we may assume $p>0$.

(1) We show the following statement:

($*$) For $f(x)\in \mathcal{M}_{c_\sigma}$, the equality $|P_{c_\sigma}(f(x))|=\max\{i\in  P_{c_\sigma}(f(x))\}$ holds if and only if $P_{c_\sigma}(f(x))=\{1\}$.

Indeed, if $P_{c_\sigma}(f(x))=\{1\}$, then $|P_{c_\sigma}(f(x))|=\max\{j\in P_{c_\sigma}(f(x))\}$. Conversely, suppose  $|P_{c_\sigma}(f(x))|=\max\{j\in P_{c_\sigma}(f(x))\}$. Then $P_{c_\sigma}(f(x))=[\,|P_{c_\sigma}(f(x))|\,]$. We write $f(x)=g(x)^j\in \mathcal{M}_{c_{\sigma}}$ with $g(x)\in R[x]$ an irreducible polynomial and $j$ the maximal power index of $g(x)$. By Lemma \ref{per}, we have $P_{c_\sigma}(f(x))=\{p^{\nu_p(\lambda_i)}\mid i\in[s], \; g(x) \mbox{ is a divisor of } x^{\lambda_i}-1 \}$. Thus it consists of $p$-powers. 
By Lemma \ref{slopop}(2)-(3), if $p\ne 2$ and $p>0$, then $P_{c_\sigma}(f(x))=\{1\}$; or if $p=2$, then $P_{c_\sigma}(f(x))=\{1,2\}$ or $P_{c_\sigma}(f(x))=\{1\}$. Hence it follows from either $p\neq 2$, or $p=2$ and $P_{c_\sigma}(f(x))\neq \{1,2\}$ that $P_{c_\sigma}(f(x))=\{1\}$.

Therefore $s(\sigma)=id$ if and only if $r(\sigma)=\sigma$ if and only if $P_{c_\sigma}(f(x))=\{1\}$ for all $f(x)\in \mathcal{M}_{c_\sigma}$ if and only if $|P_{c_\sigma}(f(x))|=\max\{j\in P_{c_\sigma}(f(x))\}$ for all $f(x)\in \mathcal{M}_{c_\sigma}$ if and only if $\mathcal{I}_{c_\sigma}=\emptyset$. This shows (1).

(2) This follows from ($*$) and Lemma \ref{per}. $\square$

\smallskip
 Given $T:=\{ n_1, n_2, \cdots, n_s\}\subset\mathbb{Z}_{>0}$ with $n_1>n_2>\cdots >n_s$ and $s\ge 2$, we denote by $L_2(T):=\{n_{s-1},n_s\}$ the last two elements of $T$, and by $L_{2nd}(T):=\{n_{s-1}\}$ the second to last element of $T$. We define a multiset $\mathcal{D}'_T:=\mathcal{D}_{T\backslash \{n_s\}}$.

For $c\in M_n(R)$ and $i\in [r_c]$, let $\tilde{\mathcal{I}}_{c,i}:=\{f(x)\in \mathcal{I}_{c,i}\mid 1\in P_c(f(x))\}$ and $\tilde{\mathcal{D}}_{c,i}:=\bigcup_{f(x)\in\tilde{\mathcal{I}}_{c,i}}\mathcal{D}_{L_2(P_c(f(x))}$. For the definitions of $\mathcal{D}_{T}$ and $\mathcal{I}_{c,i}$, we refer to Subsection \ref{number} and Section \ref{sect3}, respectively.

\begin{Lem}\label{preforcoro}
Suppose $p\neq 2$ or $p=2$ and $P_{c_\sigma}(g(x))\neq \{1,2\}$ for all $g(x)\in\mathcal{M}_{c_\sigma}$. Let $j \in [r_{c_\sigma}]$.

$(1)$ If $f(x) \in \mathcal{I}_{c_\sigma, j}$, then $\mathcal{D}_{P_{c_{s(\sigma)}}(f(x))}=\begin{cases} \mathcal{D}_{P_{c_{\sigma}}(f(x))}, & \text{ if } f(x)=x-1 \text{ or } f(x)\notin\tilde{\mathcal{I}}_{c_{\sigma},j},\\ \mathcal{D}'_{P_{c_\sigma}(f(x))}, & \text{ if } f(x)\neq x-1 \text{ and } f(x)\in\tilde{\mathcal{I}}_{c_{\sigma},j}.\end{cases}$

$(2)$ Assume $\mathcal{I}_{c_{s(\sigma)},j} = \mathcal{I}_{c_\sigma,j}$. Then

$\quad (i)$ If $x-1\in \tilde{\mathcal{I}}_{c_{\sigma},j}$, then $$\mathcal{D}_{c_{s(\sigma)},j}=\mathcal{D}_{L_2(P_{c_{\sigma}}(x-1))} \; \cup \bigcup_{x-1\neq g(x)\in\tilde{\mathcal{I}}_{c_{\sigma},j}}L_{2nd}\big(P_{c_{\sigma}}(g(x))\big) \; \cup \; (\mathcal{D}_{c_{\sigma},j}\setminus\tilde{\mathcal{D}}_{c_{\sigma},j}),$$
where the unions are taken in the sense of multisets.

$\quad (ii)$ If $x-1\notin \tilde{\mathcal{I}}_{c_{\sigma},j}$, then $$\mathcal{D}_{c_{s(\sigma)},j}=\bigcup_{g(x)\in\tilde{\mathcal{I}}_{c_\sigma,j}}L_{2nd}\big(P_{c_\sigma}(g(x))\big) \; \cup \;
(\mathcal{D}_{c_{\sigma},j}\setminus\tilde{\mathcal{D}}_{c_{\sigma},j}),$$
where the unions are taken in the sense of multisets.
\end{Lem}

{\it Proof.}
If $\mathcal{I}_{c_\sigma,j}=\emptyset$, then (1) and (2) are trivial. So we suppose $\mathcal{I}_{c_\sigma,j}\ne \emptyset$. In this case, $p>0$.

(1) This follows from Lemma \ref{sseoiff}(2) and the definition of $\mathcal{D}'_{P_{c_\sigma}(f(x))}$.

(2) By the definition of $\mathcal{D}_{c_{s(\sigma)},j}$ and the assumption $\mathcal{I}_{c_{s(\sigma)},j} = \mathcal{I}_{c_\sigma,j}$, we know $$\mathcal{D}_{c_{s(\sigma)},j}=\bigcup_{g(x)\in \mathcal{I}_{c_{s(\sigma)},j}}\mathcal{D}_{P_{c_{s(\sigma)}}(g(x))}=\bigcup_{g(x)\in \mathcal{I}_{c_{\sigma},j}}\mathcal{D}_{P_{c_{s(\sigma)}}(g(x))},$$ where the unions are taken in the sense of multisets. To complete the proof, we consider the two cases.

(i) $x-1\in \tilde{\mathcal{I}}_{c_{\sigma},j}$. According to (1), we divide the polynomials in $\mathcal{I}_{c_{\sigma},j}$ into the three parts:
\begin{align*}
&\bigcup_{g(x)\in \mathcal{I}_{c_{\sigma},j}}\mathcal{D}_{P_{c_{s(\sigma)}}(g(x))}
= \; \mathcal{D}_{P_{c_{\sigma}}(x-1)} \; \cup \bigcup_{x-1\neq g(x)\in \tilde{\mathcal{I}}_{c_\sigma,j}}\mathcal{D}'_{P_{c_{\sigma}}(g(x))} \; \cup \bigcup_{g(x)\in \mathcal{I}_{c_{\sigma},j}\backslash\tilde{\mathcal{I}}_{c_{\sigma},j}}\mathcal{D}_{P_{c_{\sigma}}(g(x))}\\
= \; & \mathcal{D}_{L_2(P_{c_{\sigma}}(x-1))} \; \cup \bigcup_{x-1\neq g(x)\in\tilde{\mathcal{I}}_{c_{\sigma},j}}L_{2nd}\big(P_{c_{\sigma}}(g(x))\big) \; \cup \; (\mathcal{D}_{c_{\sigma},j}\setminus\tilde{\mathcal{D}}_{c_{\sigma},j}),
\end{align*}
where the unions are taken in the sense of multisets.

(ii) $x-1\notin \tilde{\mathcal{I}}_{c_{\sigma},j}$. By (1), we partition the polynomials in $\mathcal{I}_{c_{\sigma},j}$ into two parts:
\begin{align*}
& \bigcup_{g(x)\in \mathcal{I}_{c_{\sigma},j}}\mathcal{D}_{P_{c_{s(\sigma)}}(g(x))}
=  \bigcup_{g(x)\in \tilde{\mathcal{I}}_{c_\sigma,j}}\mathcal{D}'_{P_{c_{\sigma}}(g(x))} \; \cup \bigcup_{g(x)\in \mathcal{I}_{c_{\sigma},j}\setminus\tilde{\mathcal{I}}_{c_\sigma,j}}\mathcal{D}_{P_{c_{\sigma}}(g(x))}\\
= & \bigcup_{g(x)\in\tilde{\mathcal{I}}_{c_\sigma,j}}L_{2nd}\big(P_{c_\sigma}(g(x))\big) \; \cup \; (\mathcal{D}_{c_{\sigma},j}\setminus\tilde{\mathcal{D}}_{c_{\sigma},j}),
\end{align*}
where the unions are taken in the sense of multisets. $\square$

\smallskip
Assume that $c_\sigma$ and $c_\tau$ are $Sg$-equivalent. Then $r_{c_\sigma}=r_{c_\tau}$ and, by definition, there is a permutation $\delta\in \Sigma_{r_{c_\sigma}}$ such that $Q_{c_\sigma,i}\simeq Q_{c_\tau,(i)\delta}$ as $R$-algebras and $\mathcal{D}_{c_\sigma,i}=\mathcal{D}_{c_\tau,(i)\delta}$ for $i\in [r_{c_\sigma}]$. Under this assumption, we have the following lemma.

\begin{Lem}\label{preforcor} Assume that $c_\sigma$ and $c_\tau$ are $Sg$-equivalent.
Suppose $p\neq 2$ or $p=2$ and $\nu_p(\lambda_i)\neq 1\neq \nu_p(\mu_{j})$ for all $i\in [s]$ and $j\in[t]$. Then the following hold.

$(1)$ For any $f(x) \in \mathcal{I}_{c_\sigma}$ with $1\in P_{c_\sigma}(f(x))$, there exists $j\in \mathbb{Z}_{>0}$ such that $p^{j}>2$ and $\{1, p^{j}\}=L_2(P_{c_\sigma}(f(x)))$.

$(2)$ Let $j \in [r_{c_\sigma}]$. Then

$\quad (i)$ $\tilde{\mathcal{I}}_{c_\sigma,j}\ne \emptyset$ if and only if $\tilde{\mathcal{I}}_{c_\tau,(j)\delta}\ne \emptyset$.

$\quad (ii)$ $\mathcal{D}_{c_{\sigma},j}\backslash\tilde{\mathcal{D}}_{c_{\sigma},j}=\mathcal{D}_{c_{\tau},(j)\delta}\backslash\tilde{\mathcal{D}}_{c_{\tau},(j)\delta}$ and $\tilde{\mathcal{D}}_{c_{\sigma},j}=\tilde{\mathcal{D}}_{c_{\tau},(j)\delta}$.

$\quad (iii)$ There is a bijection $\pi_j:\tilde{\mathcal{I}}_{c_\sigma,j}\to \tilde{\mathcal{I}}_{c_\tau,(j)\delta}$ such that $L_2\big(P_{c_\sigma}(f(x))\big)=L_2\big(P_{c_\tau}((f(x))\pi_j)\big)$ for all $f(x)\in \tilde{\mathcal{I}}_{c_\sigma,j}$. Moreover, $x-1\in \tilde{\mathcal{I}}_{c_\sigma,j}$ if and only if $x-1\in \tilde{\mathcal{I}}_{c_\tau,(j)\delta}$.  If $x-1\in \tilde{\mathcal{I}}_{c_\sigma,j}$, then we may assume $(x-1)\pi_j=x-1$.
\end{Lem}

{\it Proof.} For $p=0$, it follows from Lemma \ref{per} that all maximal elementary divisors of $c_{\sigma}$ are irreducible. This shows $\mathcal{I}_{c_{\sigma}}=\emptyset.$ So all conclusions in Lemma \ref{preforcor} are trivially true. So we assume $p>0$.

Suppose $p=2$. Then $\nu_p(\lambda_i)\neq 1$ for all $i\in [s]$ by assumption. According to Lemma \ref{per}, we know that $P_{c_\sigma}(f(x))$ consists of $p$-powers for $f(x)\in \mathcal{M}_{c_\sigma}$. This implies $2\notin P_{c_\sigma}(f(x))$ for all $f(x)\in \mathcal{M}_{c_\sigma}$, and therefore $P_{c_\sigma}(f(x))\ne\{1,2\}$ for all $f(x)\in \mathcal{M}_{c_\sigma}$. Hence, both $c_\sigma$ and $c_\tau$ satisfy the conditions of Lemma \ref{preforcoro}.

(1) Suppose $f(x) \in \mathcal{I}_{c_\sigma}$ with $1\in P_{c_\sigma}(f(x))$. In this case, we assume $i\in P_{c_\sigma}(f(x))$ such that $f(x)^i\in \mathcal{M}_{c_{\sigma}}$. It follows from $|P_{c_\sigma}(f(x))|\ne i$ that $i\ne 1$ and $|P_{c_\sigma}(f(x))|\ge 2$. Then there exists $j\in [s]$ such that $L_2(P_{c_\sigma}(f(x)))=\{1, p^{\nu_p(\lambda_j)}\}$ with $\nu_p(\lambda_j)\ne 0$. By assumption, $p\neq 2$ or $p=2$ and $\nu_p(\lambda_i)\neq 1$ for all $i\in [s]$. This implies $p^{\nu_p(\lambda_j)}>2$. Thus (1) follows.

(2)(i) Suppose $\tilde{\mathcal{I}}_{c_\sigma,j}\ne \emptyset$. Then there is $f(x)\in \mathcal{I}_{c_\sigma,j}$ such that $1\in P_{c_\sigma}(f(x))$. By (1),
there exists $j'\in \mathbb{Z}_{>0}$ such that $p^{j'}>2$ and $\{1, p^{j'}\}=L_2(P_{c_\sigma}(f(x)))$. Hence $p^{j'}-1\in \mathcal{D}_{c_\sigma,j}=\mathcal{D}_{c_\tau,(j)\delta}$
(the equality follows by the $Sg$-equivalence of $c_\sigma$ and $c_\tau$), and therefore, there is $g(x)\in \mathcal{I}_{c_\tau,(j)\delta}$ such that $p^{j'}-1=p^{k}-p^{k'}$ with
$p^{k}, p^{k'}\in P_{c_\tau}(g(x))$ for some $k>k'\in\mathbb{N}$ or $p^{j'}-1=p^{i}\in P_{c_\tau}(g(x))$ for some $i\in \mathbb{Z}_{>0}$. By Lemma \ref{slopop}, we have $j'=k$ and
$k'=0$; or $p=2, j'=1$ and $i=0$. Thanks to $2^1-1=1\notin \mathcal{D}_{c_\tau,(j)\delta}$ (see the definition of $\mathcal{D}_T$), the latter cannot occur. Hence $L_2\big(P_{c_\tau}(g(x))\big)=\{1, p^{j'}\}$ 
that is, $\tilde{\mathcal{I}}_{c_\tau,(j)\delta}\ne \emptyset$. Similarly, we can show $\tilde{\mathcal{I}}_{c_\sigma,j}\ne \emptyset$ if $\tilde{\mathcal{I}}_{c_\tau,(j)\delta}\ne \emptyset$.

(2)(ii) It follows from $\mathcal{D}_{c_{\sigma},j}=\mathcal{D}_{c_{\tau},(j)\delta}$ and the proof of (2)(i) that $\mathcal{D}_{c_{\sigma},j}\backslash\tilde{\mathcal{D}}_{c_{\sigma},j}=\mathcal{D}_{c_{\tau},(j)\delta}\backslash\tilde{\mathcal{D}}_{c_{\tau},(j)\delta}$ and $\tilde{\mathcal{D}}_{c_{\sigma},j}=\tilde{\mathcal{D}}_{c_{\tau},(j)\delta}$.

(2)(iii) By definition, $1\in P_c(f(x))$ for $f(x)\in \tilde{\mathcal{I}}_{c_{\sigma},j}$. Due to $\bigcup_{f(x)\in\tilde{\mathcal{I}}_{c_\sigma,j}}\mathcal{D}_{L_2(P_c(f(x))}=\tilde{\mathcal{D}}_{c_{\sigma},j}
=\tilde{\mathcal{D}}_{c_{\tau},(j)\delta}=\bigcup_{g(x)\in\tilde{\mathcal{I}}_{c_\tau,(j)\delta}}\mathcal{D}_{L_2(P_c(g(x))}$ and (1),  there is a bijection $\pi_j:\tilde{\mathcal{I}}_{c_\sigma,j}\to \tilde{\mathcal{I}}_{c_\tau,(j)\delta}$ such that, for $f(x)\in \tilde{\mathcal{I}}_{c_\sigma,j}$,
$$L_2\big(P_{c_\sigma}(f(x))\big)=L_2\big(P_{c_\tau}((f(x))\pi_j)\big).$$

Suppose $x-1\in \tilde{\mathcal{I}}_{c_\sigma,j}$. Then $\tilde{\mathcal{I}}_{c_\tau,(j)\delta}\ne\emptyset$ by (2)(i), and therefore there is $g(x)\in \mathcal{I}_{c_\tau,(j)\delta}$ such that $1\in P_{c_\tau}(g(x))$. It then follows from (1) that there exists $j'\in \mathbb{Z}_{>0}$ such that $p^{j'}>2$ and $\{1, p^{j'}\}=L_2(P_{c_\tau}(g(x)))$. This implies that $\nu_p(\mu_{i})=0$ and $\nu_p(\mu_{i'})=j'$ for some $i, i'\in[t]$. By Lemma \ref{per}, we have $\{1, p^{j'}\}\subseteq P_{c_\tau}(x-1)=\{p^{\nu_p(\mu_i)}\mid i\in [t]\}$. Hence $P_{c_\tau}(x-1)\backslash \{1\}\ne \emptyset$. Thanks to $p\neq 2$ or $p=2$ and $\nu_p(\mu_k)\neq 1$ for all $k\in [t]$, we get $\min\{k\in P_{c_\tau}(x-1)\backslash \{1\}\}>2$. This implies $x-1\in \mathcal{I}_{c_\tau}$. By the definition of equivalence classes of $\mathcal{I}_{c_{\tau}}$, we have $x-1\in \mathcal{I}_{c_\tau,(j)\delta}$. Since $1\in P_{c_\tau}(x-1)$, we have $x-1\in \tilde{\mathcal{I}}_{c_\tau,(j)\delta}$.  Similarly, we can show that $x-1\in \tilde{\mathcal{I}}_{c_\sigma,j}$ if $x-1\in \tilde{\mathcal{I}}_{c_\tau,(j)\delta}$.

Now we assume  $x-1\in \tilde{\mathcal{I}}_{c_\sigma,j}$ and $(x-1)\pi_j \neq x-1$. Then there are $ f(x)\in \tilde{\mathcal{I}}_{c_\sigma,j}$ and $g(x)\in \tilde{\mathcal{I}}_{c_\tau,(j)\delta}$ such that $(x-1)\pi_j=g(x)$ and $(f(x))\pi_j=x-1$. Clearly, $f(x)\ne x-1\ne g(x)$. By (1), there are $k, k'\in \mathbb{Z}_{>0}$ such that $L_2(P_{c_\sigma}(x-1))=\{1, p^k\}=L_2\big(P_{c_\tau}(g(x))\big)$ and $L_2\big(P_{c_\sigma}(f(x))\big)=\{1, p^{k'}\}=L_2(P_{c_\tau}(x-1))$. This implies that $\nu_p(\lambda_{i})=k'$ and $\nu_p(\mu_{i'})=k$ for some $i\in[s]$ and $i'\in[t]$ . By Lemma \ref{per}, we have $P_{c_\sigma}(x-1)=\{p^{\nu_p(\lambda_i)}\mid i\in [s]\}$ and $P_{c_\tau}(x-1)=\{p^{\nu_p(\mu_i)}\mid i\in [t]\}$. Thus $p^{k'}\in P_{c_\sigma}(x-1)$ and $p^k\in P_{c_\tau}(x-1)$, and therefore $k=k'$. Now we define a map:
$$\pi_j': \tilde{\mathcal{I}}_{c_\sigma,j}\lra \tilde{\mathcal{I}}_{c_\tau,(j)\delta}, \quad q(x)\mapsto \begin{cases}(q(x))\pi_j, & \mbox{ if } q(x)\not\in \{x-1,f(x)\},\\
x-1, & \mbox{ if } q(x)=x-1, \\
g(x), & \mbox{ if } q(x)= f(x).\end{cases}$$
Then $\pi'_j$ is a desired map. $\square$

\smallskip
{\bf Proof of Theorem \ref{seispop}.} Note that all algebras involved in Theorem \ref{seispop} are semisimple if $p=0$. Thus there is nothing to prove. So we may assume $p>0$.
By Theorem \ref{seocma}, it suffices to show that $c_{s(\sigma)}\stackrel{Sg}\sim c_{s(\tau)}$ if $c_{\sigma}\stackrel{Sg}\sim c_{\tau}$.

Now suppose $c_{\sigma}\stackrel{Sg}\sim c_{\tau}$. Then $r_{c_\sigma}=r_{c_\tau}$ and there is a permutation $\delta\in \Sigma_{r_{c_\sigma}}$ such that $Q_{c_\sigma,i}\simeq Q_{c_\tau,(i)\delta}$ as $R$-algebras and $\mathcal{D}_{c_\sigma,i}=\mathcal{D}_{c_\tau,(i)\delta}$ for $i\in [r_{c_\sigma}]$. Recall that $r_{c_{\sigma}}$ denotes the  number of equivalence classes of $\mathcal{I}_{c_{\sigma}}$ (see Section \ref{sect3}). If $p=2$, then $\nu_p(\lambda_i)\neq 1$ for all $i\in [s]$. By Lemma \ref{per}, $P_{c_\sigma}(f(x))$ does not contain $2$ for all $f(x)\in \mathcal{M}_{c_\sigma}$, and therefore $P_{c_\sigma}(f(x))\ne\{1,2\}$ for all $f(x)\in \mathcal{M}_{c_\sigma}$. Thus both $c_\sigma$ and $c_\tau$ satisfy the conditions of Lemmas \ref{sseoiff}-\ref{preforcor}. It then follows from Lemma \ref{sseoiff}(2) that $\mathcal{I}_{c_{s(\sigma)}}=\mathcal{I}_{c_\sigma}$ and $\mathcal{I}_{c_{s(\tau)}}=\mathcal{I}_{c_\tau}$. Now we consider the two cases.

(1) $s(\sigma)=id$. Then $\mathcal{I}_{c_\sigma}=\emptyset$ by Lemma \ref{sseoiff}(1). Hence $r_{c_{\sigma}}=0$, and therefore  $r_{c_{\tau}}=r_{c_{\sigma}}=0$. Thus $\mathcal{I}_{c_\tau}=\emptyset$ and $\mathcal{I}_{c_{s(\tau)}}=\emptyset$.
By Lemma \ref{sseoiff}(1), we get $s(\tau)=id$. Hence $c_{s(\sigma)}\stackrel{Sg}\sim c_{s(\tau)}$.

(2) $s(\sigma)\neq id$. Then $\mathcal{I}_{c_\sigma}\ne\emptyset$ by Lemma \ref{sseoiff}(1). Thus $r_{c_{\sigma}}\ne 0$ and $r_{c_{s(\tau)}}=r_{c_\tau}=r_{c_\sigma}=r_{c_{s(\sigma)}}\neq 0$. Thus $s(\tau)\neq id$. In this case, let $r:=r_{c_{s(\sigma)}}$ and assume that $\mathcal{I}_{c_{s(\sigma)},i} = \mathcal{I}_{c_\sigma,i}$ and $\mathcal{I}_{c_{s(\tau)},i} = \mathcal{I}_{c_\tau,i}$ for all $i\in [r]$. We already know $Q_{c_{s(\sigma)},i}\simeq Q_{c_{s(\tau)},(i)\delta}$ as algebras for $i\in [r]$. It remains to show $\mathcal{D}_{c_{s(\sigma)},i}=\mathcal{D}_{c_{s(\tau)},(i)\delta}$ for $i\in [r]$. For this purpose we consider the following two cases:

(i) $x-1\in\tilde{\mathcal{I}}_{c_\sigma,i}$. Then it follows from Lemma \ref{preforcor}(2)(iii) that $x-1\in \tilde{\mathcal{I}}_{c_\tau,(i)\delta}$ and there is a bijection $\pi_i:\tilde{\mathcal{I}}_{c_\sigma,i}\to \tilde{\mathcal{I}}_{c_\tau,(i)\delta}$ such that $(x-1)\pi_i=x-1$ and $L_2\big(P_{c_\sigma}(f(x))\big)=L_2\big(P_{c_\tau}((f(x))\pi_i)\big)$ for all $f(x)\in \tilde{\mathcal{I}}_{c_\sigma,i}$. According to Lemmas \ref{preforcoro}(2)(i) and \ref{preforcor}(2)(ii), in order to prove $\mathcal{D}_{c_{s(\sigma)},i}=\mathcal{D}_{c_{s(\tau)},(i)\delta}$, it suffices to show:
$$(\sharp)\quad \mathcal{D}_{L_2\big(P_{c_\sigma}(x-1)\big)} \, \cup\bigcup_{x-1\neq h(x)\in\tilde{\mathcal{I}}_{c_\sigma,i}}L_{2nd}\big(P_{c_\sigma}(h(x))\big)=\mathcal{D}_{L_2\big(P_{c_\tau}(x-1)\big)}
\, \cup \bigcup_{x-1\neq g(x)\in\tilde{\mathcal{I}}_{c_\tau,(i)\delta}}L_{2nd}\big(P_{c_\tau}(g(x))\big).$$

In fact, $\pi_i$ gives rise to the equality $L_2\big(P_{c_\sigma}(x-1)\big)=L_2\big(P_{c_\tau}(x-1)\big)$. Thus $\mathcal{D}_{L_2\big(P_{c_\sigma}(x-1)\big)}=\mathcal{D}_{L_2\big(P_{c_\tau}(x-1)\big)}$.
Moreover, due to $L_2\big(P_{c_\sigma}(h(x))\big)=L_2\big(P_{c_\tau}((h(x)\big)\pi_i))$ for $h(x)\in \tilde{\mathcal{I}}_{c_\sigma,i}\setminus \{x-1\}$, we have
$$\bigcup_{x-1\neq h(x)\in\tilde{\mathcal{I}}_{c_\sigma,i}}L_{2nd}\big(P_{c_\sigma}(h(x))\big) \; =\bigcup_{x-1\neq g(x)\in\tilde{\mathcal{I}}_{c_\tau,(i)\delta}}L_{2nd}\big(P_{c_\tau}(g(x))\big).$$ This shows $(\sharp)$.

(ii) $x-1\notin\tilde{\mathcal{I}}_{c_\sigma,i}$. In this case, by Lemma \ref{preforcor}(2)(iii), $x-1\notin \tilde{\mathcal{I}}_{c_\tau,(i)\delta}$ and there is a bijection $\pi_i:\tilde{\mathcal{I}}_{c_\sigma,i}\to \tilde{\mathcal{I}}_{c_\tau,(i)\delta}$ such that $L_2\big(P_{c_\sigma}(f(x))\big)=L_2\big(P_{c_\tau}((f(x))\pi_i)\big)$ for all $f(x)\in \tilde{\mathcal{I}}_{c_\sigma,i}$. Now, we prove $\mathcal{D}_{c_{s(\sigma)},i}=
\mathcal{D}_{c_{s(\tau)},(i)\delta}$. By Lemmas \ref{preforcoro}(2)(ii) and \ref{preforcor}(2)(ii), it suffices to show
$$\bigcup_{ h(x)\in\tilde{\mathcal{I}}_{c_\sigma,i}}L_{2nd}\big(P_{c_\sigma}(h(x))\big) \;
 =\bigcup_{h'(x)\in\tilde{\mathcal{I}}_{c_\tau,(i)\delta}}L_{2nd}\big(P_{c_\tau}(h'(x))\big).$$ But this follows immediately from the bijection $\pi_i$.

Thus we have shown $\mathcal{D}_{c_{s(\sigma)},i}=\mathcal{D}_{c_{s(\tau)},(i)\delta}$ for all $i\in[r]$, and therefore $c_{s(\sigma)}\stackrel{Sg}\sim c_{s(\tau)}$. $\square$

\medskip
The example shows that some assumptions in Theorem \ref{seispop} cannot be removed.

\begin{Bsp}\label{cnr}{\rm
Let $R=\mathbb{F}_2$ be the field of two elements, $\sigma\in \Sigma_9$ of the cycle type $(6, 3)$, and $\tau\in \Sigma_3$ of the cycle type $(3)$. Then $\nu_2(6)=1, \nu_2(3)=0, \mathcal{E}_{c_\sigma}=\{x-1,(x-1)^2,x^2+x+1,(x^2+x+1)^2\}$, $P_{c_\sigma}((x-1)^2)=\{1,2\}=P_{c_\sigma}((x^2+x+1)^2)$, $|P_{c_\sigma}((x-1)^2)|=2=\max \{i\in P_{c_\sigma}((x-1)^2)\}$, and $|P_{c_\sigma}((x^2+x+1)^2)|=2=\max \{i\in P_{c_\sigma}((x^2+x+1)^2)\}$. Thus $\mathcal{I}_{c_\sigma}=\emptyset$. Clearly, $s(\sigma)\in \Sigma_9$ is of the cycle type $(6, 1, 1, 1)$ and $\mathcal{E}_{c_{s(\sigma)}}=\{x-1,(x-1)^2,(x^2+x+1)^2\}$ and $\mathcal{I}_{c_{s(\sigma)}}=\{x^2+x+1\}$.

By calculations, $\mathcal{E}_{c_{\tau}}=\{x-1,x^2+x+1\}$, $\mathcal{I}_{c_{\tau}}=\emptyset$, $\mathcal{E}_{c_{s(\tau)}}=\{x-1\}$ and $\mathcal{I}_{c_{s(\tau)}}=\emptyset$. Hence $S_9(c_{s(\sigma)},R)$ and $S_3(c_{s(\tau)},R)$ are not singularly equivalent, while $S_9(c_{\sigma},R)$ and $S_3(c_{\tau},R)$ are singularly equivalent by Theorem \ref{seocma}. Thus Theorem \ref{seispop} may be false without the requirement on $\nu_p(\lambda_i)$ and $\nu_p(\mu_j)$.
}
\end{Bsp}

Note that $S_n(c_\sigma,R)$ and $S_m(c_\tau,R)$ are Morita equivalent if and only if they are derived equivalent (see \cite[Corollary 1.3]{lx1}).
If $\sigma$ and $\tau$ are $p$-singular permutations for a prime $p>0$, then $S_n(c_{\sigma},R)$ and $S_m(c_{\tau},R)$ are stably equivalent if and only if they are Morita equivalent (see \cite[Corollary 4.20]{lx2}). The next example demonstrates that singular equivalences are substantially different from Morita, derived and stable equivalences even for the centralizer algebras of permutation matrices.

\begin{Bsp}\label{dfsefpm}{\rm
Let $R$ be an algebraically closed field of  characteristic $p\ge 11$, $n=7p+15p^2+7p^3, m=8p+15p^2+7p^3$. We take $\sigma\in \Sigma_{n}$ of the cycle type $(7p^3, 7p^2, 5p^2,3p^2, 7p)$, and $\tau\in \Sigma_{m}$ of the cycle type $(7p^3, 7p^2, 5p^2, 3p^2, 5p, 3p)$. Then both $\sigma=s(\sigma)$ and $\tau = s(\tau)$ are $p$-singular. We will show that $S_{n}(c_\sigma,R)$ and $S_{m}(c_\tau,R)$ are singularly equivalent, but not Morita equivalent.

Let $\zeta_i$ be a primitive $i$-th root of unity over $R$ for $i\in\{3,5,7\}$.
By Lemma \ref{per}, we have
$\mathcal{E}_{c_\sigma}=\{(x-1)^{p}, (x-1)^{p^2}, (x-1)^{p^3},(x-\zeta_3)^{p^2},(x-\zeta_3^2)^{p^2},(x-\zeta_{5})^{p^2},\cdots,(x-\zeta_{5}^4)^{p^2},(x-\zeta_{7})^{p},\cdots,(x-\zeta_{7}^6)^{p},
(x-\zeta_{7})^{p^2},\cdots,(x-\zeta_{7}^6)^{p^2},(x-\zeta_{7})^{p^3},\cdots,(x-\zeta_{7}^6)^{p^3}\}$ and $\mathcal{E}_{c_\tau}=\{(x-1)^{p}, (x-1)^{p^2}, (x-1)^{p^3},(x-\zeta_3)^{p},(x-\zeta_3^2)^{p},(x-\zeta_3)^{p^2},(x-\zeta_3^2)^{p^2},(x-\zeta_{5})^{p},\cdots,(x-\zeta_{5}^4)^{p},(x-\zeta_{5})^{p^2},
\cdots,(x-\zeta_{5}^4)^{p^2},(x-\zeta_{7})^{p^2},\cdots,(x-\zeta_{7}^6)^{p^2},(x-\zeta_{7})^{p^3},\cdots,(x-\zeta_{7}^6)^{p^3}\}$. It follows that $$\mathcal{M}_{c_\sigma}=\{(x-1)^{p^3},(x-\zeta_3)^{p^2},(x-\zeta_3^2)^{p^2},(x-\zeta_{5})^{p^2},\cdots,(x-\zeta_{5}^4)^{p^2},(x-\zeta_{7})^{p^3},
\cdots,(x-\zeta_{7}^6)^{p^3}\}=\mathcal{M}_{c_\tau},$$ $P_{c_\sigma}((x-1)^{p^3})=\{p,p^2,p^3\}=P_{c_\sigma}((x-\zeta_7^k)^{p^3}), P_{c_\sigma}((x-\zeta_3^i)^{p^2})=P_{c_\sigma}((x-\zeta_5^j)^{p^2})=\{p^2\}$ for $i\in[2], j\in [4], k\in [6]$; $P_{c_\tau}((x-1)^{p^3})=\{p,p^2,p^3\}, P_{c_\tau}((x-\zeta_3^i)^{p^2})=P_{c_\tau}((x-\zeta_5^j)^{p^2})=\{p,p^2\}, P_{c_\tau}((x-\zeta_7^k)^{p^3})=\{p^2,p^3\}$ for $i\in[2], j\in [4], k\in [6]$.

Therefore, $\mathcal{I}_{c_\sigma}=\{x-1,x-\zeta_3,x-\zeta_3^2,x-\zeta_{5},\cdots,x-\zeta_{5}^4,x-\zeta_{7},\cdots,x-\zeta_{7}^6\}=\mathcal{I}_{c_\tau}$. Since $R$ is an algebraically closed field, we know  $r_{c_\sigma}=1=r_{c_\tau}$and $\mathcal{D}_{c_\sigma,1}=\mathcal{D}_{c_\tau,1}$. Thus $c_\sigma$ and $c_\tau$ are $Sg$-equivalent by Definition \ref{newequrel}(2), and therefore $S_{n}(c_\sigma,R)$ and $S_{m}(c_\tau,R)$ are singularly equivalent by Theorem \ref{seocma}. Since $P_{c_\sigma}((x-\zeta_3)^{p^2})\neq P_{c_\tau}(f(x))$ holds for $f(x)\in \mathcal{M}_{c_\tau}$, it follows tht $c_\sigma$ and $c_\tau$ are not $M$-equivalent by Definition \ref{newequrel2}(1), and therefore $S_{n}(c_\sigma,R)$ and $S_{m}(c_\tau,R)$ are not Morita equivalent by \cite[Theorem 1.1]{lx1}.
}
\end{Bsp}

Finally, we point out a special case where Morita, derived, stable and singular equivalences are mutually implied by each other.
We first prepare a couple of lemmas.

\textbf{For the rest of this subsection}, let $\sigma\in\Sigma_n$ be of the cycle type $(\lambda_1, \cdots, \lambda_s)$ and $\sigma^+\in\Sigma_{n+\lambda_{s+1}}$ be of the cycle type $\lambda^+:=(\lambda_1, \cdots, \lambda_s, \lambda_{s+1})$. We wirte $\lambda_i = p^{\nu_p(\lambda_i)}\lambda_i'$ with $\nu_p(\lambda_i')=0$ for all $i\in [s+1]$. Set $$I:=\{j\in[s]\mid \lambda_{s+1}' \mbox{ divides } \lambda_j'\} \; \mbox{  and } \; J:=\{p^{\nu_p(\lambda_i)}\mid i\in I\}.$$

\begin{Lem}\label{preone-more}
There exists an irreducible factor $f(x)$ of $x^{\lambda_{s+1}'}-1$ such that $P_{c_\sigma}(f(x))=J$ and $P_{c_{\sigma^+}}(f(x))=J\cup \{p^{\nu_p(\lambda_{s+1})}\}$.
\end{Lem}

{\it Proof.}
If $I=[s]$, then $P_{c_\sigma}(x-1)=J$ and $P_{c_{\sigma^+}}(x-1)=J\cup \{p^{\nu_p(\lambda_{s+1})}\}$, according to Lemma \ref{per}. If $I\neq[s]$, there exists $i\in [s]$ such that $\lambda_{s+1}'\nmid \lambda_i'$, and therefore $|\Phi|\ge 2$, where we define $\Phi:=\{\lambda_j'\mid j\in [s+1]\setminus I\}$. By Lemma \ref{no-2}(1), there is an irreducible factor $f(x)$ of $x^{\lambda_{s+1}'}-1$ such that $f(x)\nmid x^{\lambda_j'}-1$ for all $j\in [s]\setminus I$. For $j\in I$, we have $\lambda'_{s+1}\mid \lambda'_j$, it follows from Lemma \ref{no-2}(2) that $x^{\lambda'_{s+1}}-1\mid x^{\lambda'_j}-1$ and $f(x)\mid x^{\lambda_j'}-1$ . Then we deduce from Lemma \ref{per} that $P_{c_\sigma}(f(x))=J$ and $P_{c_{\sigma^+}}(f(x))=J\cup \{p^{\nu_p(\lambda_{s+1})}\}$. $\square$

\medskip
Recall that, for $c\in M_n(R)$, we define $\mathcal{U}_c:=\bigcup_{g(x)\in \mathcal{M}_c}\mathcal{D}_{P_c(g(x))}$, where the union is taken in the sense of multisets. Clearly, $\mathcal{U}_c=\bigcup_{f(x)\in \mathcal{I}_c}\mathcal{D}_{P_c(f(x))}$.

\begin{Koro}\label{one-more} Suppose that  $R$ is of characteristic $p\ge 0$, $I\neq \emptyset$ and one of the three conditions holds:

$(i)$ $p\neq 2$;

$(ii)$ $p=2$ and $\nu_p(\lambda_{s+1}) \neq 1$; 

$(iii)$ $p=2$ and $\nu_p(\lambda_j) > 0$ for some $j\in I$.

\noindent Then the following are equivalent:

$(1)$ $S_n(c_\sigma,R)$ and $S_{n+\lambda_{s+1}}(c_{\sigma^+},R)$ are Morita equivalent.

$(2)$ $S_n(c_\sigma,R)$ and $S_{n+\lambda_{s+1}}(c_{\sigma^+},R)$ are derived equivalent.

$(3)$ $S_n(c_\sigma,R)$ and $S_{n+\lambda_{s+1}}(c_{\sigma^+},R)$ are stably equivalent.

$(4)$ $S_n(c_\sigma,R)$ and $S_{n+\lambda_{s+1}}(c_{\sigma^+},R)$ are singularly equivalent.

$(5)$ There exists $k\in [s]$ such that $\nu_p(\lambda_k)=\nu_p(\lambda_{s+1})$ and $\lambda_{s+1}'\mid\lambda_k'$.
\end{Koro}

{\it Proof.} Clearly, (1) $\Ra$ (2) $\Ra$ (4) and (1) $\Ra$ (3) $\Ra$ (4) by Theorem \ref{seocma} and Remark \ref{rber}.

(5) $\Ra$ (1) Suppose that there exists $j\in [s]$ such that $\nu_p(\lambda_j)=\nu_p(\lambda_{s+1})$ and $\lambda_{s+1}' \mid \lambda_j'$. Then $\lambda_{s+1}\mid\lambda_j$. Thus $x^{\lambda_{s+1}}-1\mid x^{\lambda_j}-1$ by Lemma \ref{no-2}(2). It follows from Lemma \ref{per} that $\mathcal{E}_{c_{\sigma^+}}=\mathcal{E}_{c_\sigma}$. Hence $c_\sigma$ and $c_{\sigma^+}$ are $M$-equivalent, and therefore $S_n(c_\sigma,R)$ and $S_{n+\lambda_{s+1}}(c_{\sigma^+},R)$ are Morita equivalent \cite[Theorem 1.1]{lx1}.

Note that (5) always holds if $p=0$. In fact, in this case, $\nu_p(\lambda_j)=0$ for all $j\in [s+1]$. By assumption, $I\neq \emptyset$, that is, $\lambda_{s+1}'\mid \lambda_t'$ for some $t\in [s]$. Therefore (5) holds.

In the following \textbf{we assume $p>0$} and prove that (4) implies (5).

By Theorem \ref{seocma}, we may suppose that $c_\sigma$ and $c_{\sigma^+}$ are $Sg$-equivalent. Then
$\mathcal{U}_{c_\sigma}=\mathcal{U}_{c_{\sigma^+}}$ and  $\#_n(\mathcal{U}_{c_\sigma})=\#_n(\mathcal{U}_{c_{\sigma^+}})$ for all $n\in \mathbb{N}$.

Contrarily, suppose that, for any $k\in[s]$, we have $\nu_p(\lambda_k)\ne \nu_p(\lambda_{s+1})$, or $\lambda_{s+1}'\nmid \lambda_k'$. In particular, for $i\in I\subseteq [s]$, we have $\lambda_{s+1}'\mid \lambda_i'$, and therefore $\nu_p(\lambda_i)\neq \nu_p(\lambda_{s+1})$. We write $J=\{p^{\nu_p(\lambda_i)}\mid i\in I\}=\{ p^{m_1}, p^{m_2}, \cdots, p^{m_\ell}\}$ with $m_1>m_2>\cdots >m_\ell\ge 0$. Then $J\ne \emptyset$ by $I\neq \emptyset$, and $\nu_p(\lambda_{s+1})\neq m_i$ for all $i\in [\ell]$, that is, $p^{\nu_p(\lambda_{s+1})}\notin J$.

Note that $x^{\lambda_{i}}-1=(x^{\lambda_{i}'}-1)^{p^{\nu_p(\lambda_{i})}}$ for all $i\in [s+1]$. By Lemma \ref{per}, we have
\begin{align*}
\mathcal{E}_{c_{\sigma}} = \{h(x)^{p^{\nu_p(\lambda_i)}}\mid i\in[s], \; h(x)~\mbox{is an irreducible factor of}~ x^{\lambda_i'}-1 \}\ \text{ and }\\
\mathcal{E}_{c_{\sigma^+}}=\mathcal{E}_{c_\sigma}\cup \{h(x)^{p^{\nu_p(\lambda_{s+1})}}\mid h(x) \mbox{ is an irreducible factor of } x^{\lambda_{s+1}'}-1\}.
\end{align*} Particularly, both $P_{c_\sigma}(f(x))$ and $P_{c_{\sigma^+}}(g(x))$ consist of $p$-powers for $f(x)\in \mathcal{M}_{c_\sigma}$ and $g(x)\in \mathcal{M}_{c_{\sigma^+}}$.

For any $i\in [s+1]$ and any irreducible factor $f(x)$ of $x^{\lambda_{i}'}-1$, we have
$$(\heartsuit)\quad P_{c_{\sigma^+}}(f(x))=\begin{cases}P_{c_\sigma}(f(x)), & \text{ if } f(x)\nmid x^{\lambda_{s+1}'}-1,\\
P_{c_\sigma}(f(x))\cup\{p^{\nu_p(\lambda_{s+1})}\}, & \text{ if } f(x)\mid x^{\lambda_{s+1}'}-1. \end{cases}$$

By $(\heartsuit)$, we know that the difference between $\mathcal{U}_{c_\sigma}$ and $\mathcal{U}_{c_{\sigma^+}}$ is determined solely by the sets of power indices of the irreducible factors of $x^{\lambda_{s+1}'}-1$. Now, we choose an irreducible factor $f(x)$ of $x^{\lambda_{s+1}'}-1$ as in Lemma \ref{preone-more}, that is, $P_{c_\sigma}(f(x))=J$ and $P_{c_{\sigma^+}}(f(x))=J\cup \{p^{\nu_p(\lambda_{s+1})}\}$. We then compare $p^{\nu_p(\lambda_{s+1})}$ with $p^{m_\ell}$ defined by $J$ and apply Lemma \ref{preofpp} to obtain a contradiction in each of cases (i) -- (iii).

Suppose $m_\ell>\nu_p(\lambda_{s+1})$. By Lemma \ref{preofpp}(2), we have  $\#_{p^{m_\ell}}(\mathcal{D}_{P_{c_{\sigma}}(f(x))})>\#_{p^{m_\ell}}(\mathcal{D}_{P_{c_{\sigma^+}}(f(x))})$. Comparing $\#_{p^{m_\ell}}(\mathcal{U}_{c_\sigma})$ with $\#_{p^{m_\ell}}(\mathcal{U}_{c_{\sigma^+}})$, we know from $(\heartsuit)$ and Lemma \ref{preofpp}(2) that $\#_{p^{m_\ell}}(\mathcal{U}_{c_{\sigma}})>\#_{p^{m_\ell}}(\mathcal{U}_{c_{\sigma^+}})$. This contradicts $\mathcal{U}_{c_\sigma}=\mathcal{U}_{c_{\sigma^+}}$.

Thus $m_\ell\le\nu_p(\lambda_{s+1})$. Due to $\nu_p(\lambda_{s+1})\ne m_i$ for $i\in[\ell]$, we have $m_\ell<\nu_p(\lambda_{s+1})$.
Then either $ m_1<\nu_p(\lambda_{s+1})$, or there is $i\in [\ell]$ with $i\ge 2$, such that  $m_i <\nu_p(\lambda_{s+1})< m_{i-1}$. In both cases, we have  $p^{\nu_p(\lambda_{s+1})}-p^{m_i}\ge 1$, with $i=\min\{i\in [\ell]\mid \nu_p(\lambda_{s+1})\ge m_i\}$.

Assume $p^{\nu_p(\lambda_{s+1})}-p^{m_i}>1$. Then $\#_{p^{\nu_p(\lambda_{s+1})}-p^{m_i}}(\mathcal{D}_{P_{c_{\sigma^+}}(f(x))})>\#_{p^{\nu_p(\lambda_{s+1})}-p^{m_i}}(\mathcal{D}_{P_{c_{\sigma}}(f(x))})$ by Lemma \ref{preofpp}(1). Comparing $\#_{p^{\nu_p(\lambda_{s+1})}-p^{m_i}}(\mathcal{U}_{c_\sigma})$ with $\#_{p^{\nu_p(\lambda_{s+1})}-p^{m_i}}(\mathcal{U}_{c_{\sigma^+}})$, we know from $(\heartsuit)$ and Lemma \ref{preofpp}(1) that $\#_{p^{\nu_p(\lambda_{s+1})}-p^{m_i}}(\mathcal{U}_{c_{\sigma^+}})>\#_{p^{\nu_p(\lambda_{s+1})}-p^{m_i}}(\mathcal{U}_{c_\sigma})$. This again contradicts $\mathcal{U}_{c_\sigma}=\mathcal{U}_{c_{\sigma^+}}$.

Thus $p^{\nu_p(\lambda_{s+1})}-p^{m_i}=1$. By Lemma \ref{slopop}(3), we have $p=2, \nu_p(\lambda_{s+1})=1$ and $m_i=0$. This shows trivially that both (i) and (ii) lead to a contradiction. Thus we assume the case (iii), that is, $p=2$ and $\nu_p(\lambda_j)>0$ for some $j\in I$. Then $m_{1}>m_i=0$ and $i\ge 2$. Thanks to $1=\nu_p(\lambda_{s+1})\ne m_j$ for all $j\in [\ell]$ and $m_{i-1}>m_{i}=0$, we have $m_{i-1} \ge 2 > \nu_p(\lambda_{s+1}) =1 > m_{i}=0$. By Lemma \ref{preofpp}(3), we have $\#_{p^{m_{i-1}}-p^{m_i}}(\mathcal{D}_{P_{c_\sigma}(f(x))})>\#_{p^{m_{i-1}}-p^{m_i}}(\mathcal{D}_{P_{c_{\sigma^+}}(f(x))})$. Comparing $\#_{p^{m_{i-1}}-p^{m_i}}(\mathcal{U}_{c_\sigma})$ with $\#_{p^{m_{i-1}}-p^{m_i}}(\mathcal{U}_{c_{\sigma^+}})$, we know from $(\heartsuit)$ and Lemma \ref{preofpp}(3) that $\#_{p^{m_{i-1}}-p^{m_i}}(\mathcal{U}_{c_{\sigma}})>\#_{p^{m_{i-1}}-p^{m_i}}(\mathcal{U}_{c_{\sigma^+}})$. This again contradicts $\mathcal{U}_{c_\sigma}=\mathcal{U}_{c_{\sigma^+}}$, and therefore (5) follows. $\square$

\smallskip
Remark that the condition $I\neq \emptyset$ in Corollary \ref{one-more} is fulfilled if one takes $\lambda_{s+1}=1$. In the next lemma, we show when $I\neq \emptyset$ holds.

\begin{Lem}\label{caone-more}
If $\mathcal{U}_{c_\sigma}=\mathcal{U}_{c_{\sigma^+}}$ and $\nu_p(\lambda_{s+1})>0$, then $I\neq \emptyset$.
\end{Lem}

{\it Proof.} We suppose contrarily that $I= \emptyset$ holds. Then $J=\emptyset$. By Lemma \ref{preone-more}, there is an irreducible factor $f(x)$ of $x^{\lambda_{s+1}'}-1$ such that $P_{c_\sigma}(f(x))=\emptyset$ and $P_{c_{\sigma^+}}(f(x))=\{p^{\nu_p(\lambda_{s+1})}\}$. As $P_{c_\sigma}(f(x))=\emptyset$, we have $f(x)^i\notin \mathcal{E}_{c_\sigma}$ for all $i\in \mathbb{N}$. It follows from $\nu_p(\lambda_{s+1})>0$ that $p>0$, $p^{\nu_p(\lambda_{s+1})}>1$ and $p^{\nu_p(\lambda_{s+1})}\in \mathcal{D}_{P_{c_{\sigma^+}}(f(x))}$. Now, comparing $\#_{p^{\nu_p(\lambda_{s+1})}}(\mathcal{U}_{c_\sigma})$ with $\#_{p^{\nu_p(\lambda_{s+1})}}(\mathcal{U}_{c_{\sigma^+}})$, we have $\#_{p^{\nu_p(\lambda_{s+1})}}(\mathcal{U}_{c_{\sigma^+}})>\#_{p^{\nu_p(\lambda_{s+1})}}(\mathcal{U}_{c_\sigma})$ by Lemma \ref{preofpp}(1) and $(\heartsuit)$. This contradicts the assumption $\mathcal{U}_{c_\sigma}=\mathcal{U}_{c_{\sigma^+}}$, and shows $I\ne \emptyset$. $\square$

\section{Homological conjectures and invariants of singular equivalences\label{sect5.2}}
In this section, we prove that the Cartan determinant and Auslander--Reiten conjectures hold true for centralizer matrix algebras. Our results in this section, together with known results, show that all homological conjectures hold true for centralizer matrix algebras. Also, some homological invariants are given for singular equivalences of centralizer matrix algebras.

Let $\Lambda$ be an Artin algebra.
If $\gd(\Lambda)<\infty$, then $\det(C_\Lambda)=\pm 1$ (see \cite{e1}). Furthermore, a well-known conjecture, called the Cartan determinant conjecture (see \cite{z1}), excludes the case: $\det(C_\Lambda)=-1$.

\smallskip
\emph{Cartan determinant conjecture} (CDC): If $\gd(\Lambda)<\infty$, then $\det(C_\Lambda)=1$.

\smallskip
This conjecture is known in few cases only. For example, if $\Lambda$ is a graded algebra or quasi-hereditary algebras or algebras of radical-cube-zero, then (CDC) holds for $\Lambda$ (see \cite{wilson, bf, z1, FZ}).

We mention another not yet solved homological conjecture.

\smallskip
\emph{Auslander--Reiten/generalized Nakayama conjecture} (ARC/GNC): If $M\in \Lambda\modcat$ is a self-orthogonal generator, then $M$ is projective (see \cite{ar}).

\smallskip
In general, all of these conjectures are open to date.
We will, however, verify these conjectures for centralizer matrix algebras. First, we have the following simple observation.

\begin{Prop}\label{ar-for-cma}
(ARC/GNC) holds for centralizer matrix algebras.
\end{Prop}

{\it Proof.} Note that centralizer matrix algebras are a CM-finite Gorenstein algebras (see \cite[Theorem 1.1(2)]{xz2} and \cite[Theorem 1.1]{kiwy1}). It was shown in \cite[Corollary 3.6]{z2} that (ARC/GNC) holds true for CM-finite Gorenstein algebras. Thus (ARC/GNC) holds true for centralizer matrix algebras. $\square$

\begin{Rem}{\rm It was proved that the finitistic dimension conjecture (FDC) \cite{bass} is true for centralizer matrix algebras over fields (see \cite{lx1, xz2}).
This implies that the strong Nakayama \cite{cf1}, generalized Nakayama \cite{ar}, Nakayama \cite{naka}, Wakamatsu tilting \cite[Section IV.3, p.71]{br} and
tilting (or projective) complement conjectures \cite{hu} hold true for centralizer matrix algebras over fields. Also, it is known that if (FDC) holds for all algebras,
then (ARC/GNC) holds for all algebras. Clearly, Gorenstein symmetry conjecture holds for centralizer matrix algebras. Moreover, Tachikawa's first and second conjectures
(TC1) and (TC2) (see \cite[p.115-116]{tachikawa}) hold true for centralizer matrix algebras. In fact, we consider the generator and cogenerator $M:=A\oplus D(A)$
for a centralizer matrix algebra $A$. Then it follows from the condition of (TC1) and Proposition \ref{ar-for-cma} that $M$ is projective, and therefore $D(A)$ is projective
and $A$ is self-injective. This implies that (TC1) holds true. The validity of (TC2) follows from the following general fact: if (ARC) holds for a class of algebras,
then (TC2) holds for all self-injective algebras in this class.
Altogether, all homological conjectures in \cite[Conjectures, p. 409]{ars} hold true for centralizer matrix algebras over fields.
}\end{Rem}

\medskip
Next, we study the quasi-heredity of centralizer matrix algebras. First, we recall the definition of quasi-hereditary algebras introduced in \cite{cps}.

\begin{Def}{\rm \cite{cps}}\label{qha}
Let $A$ be a finite-dimensional algebra over a field $R$. An ideal $J$ in $A$ is called a \emph{heredity ideal} if $J$ is idempotent, $J \rad(A) J = 0$ and $J$ is a projective left (or right) $A$-module. The algebra $A$ is said to be \emph{quasi-hereditary} provided there is a finite chain $0=J_0\subseteq J_1\subseteq J_2\subseteq\cdots\subseteq J_n = A$ of ideals in $A$ such that $J_j/J_{j-1}$ is a heredity ideal in $A/J_{j-1}$ for all $j$.
\end{Def}

Quasi-hereditary algebras were motivated by describing the highest weight category of semisimple Lie algebras. For more information on quasi-hereditary algebras, we refer to a recent survey \cite{xicc} and the references therein. The following lemma generalizes a result in~\cite[Theorem 1.1(2)]{xz-LAA}.

\begin{Lem}\label{efqha} For $c\in M_n(R)$, let $S:=S_n(c,R)$ and $C$ be the Cartan matrix of $S$. Then

$(1)$ $\det(C)=\begin{cases} \prod_{f(x)\in \mathcal{I}_c}\prod_{i\in \mathcal{D}_{P_c(f(x))}}i, & \text{ if } \mathcal{I}_c\ne\emptyset,\\
1, & \text{ otherwise. } \end{cases}$ \\

$(2)$ The following are equivalent:

$\quad (i)$ $S$ is quasi-hereditary.

$\quad (ii)$ $\gd(S)<\infty$.

$\quad (iii)$ $\gd(S)\le 2$.

$\quad (iv)$ $\mathcal{I}_c=\emptyset$.

$\quad (v)$ $|P_c(f(x))|=\max\{j\in P_c(f(x))\}$ for all $f(x)\in \mathcal{M}_c$.

$\quad (vi)$ $\det(C)=1$.
\end{Lem}

{\it Proof.} Following the notation in Section \ref{sec4}, we have $S\simeq \prod^{l_c}_{i=1}{\End}_{U_i}(M_i)$ as $R$-algebras, where $U_i$ denotes the algebra $R[x]/(f_i(x)^{n_i})$, $\mathcal{B}(M_i)\simeq  \bigoplus_{r\in {P_c(f_i(x))}} R[x]/(f_i(x)^r)$ is the basic $U_i$-module of $M_i$, and $l_c$ is the number of maximal elementary divisors of $c$. Thus $\det(C_{\End_{U_i}(M_i)})=\det(C_{\End_{U_i}(\mathcal{B}(M_i))})$ and $\gd(\End_{U_i}(M_i))=\gd(\End_{U_i}(\mathcal{B}(M_i)))$.

(1) By Lemma \ref{cmoe}, $\det(C_{\End_{U_i}(\mathcal{B}(M_i))})=\prod_{j\in \mathcal{H}_{P_c(f_i(x)^{n_i})}}j$ for $i\in [l_c]$. It follows from the definitions of $\mathcal{D}_{P_c(f_i(x)^{n_i})}$ and $\mathcal{I}_c$ that $$\det(C_{\End_{U_i}(M_i)})=\det(C_{\End_{U_i}(\mathcal{B}(M_i))})=\begin{cases} \prod_{j\in \mathcal{D}_{P_c(f_i(x))}}j, & \text{ if } f_i(x)\in\mathcal{I}_c,\\
1, & \text{ otherwise. } \end{cases}$$ Thus (1) follows.

(2) By Corollary \ref{gdoeog} and the definition of $\mathcal{I}_c$, $\gd(S)<\infty$ if and only if $\gd(\End_{U_i}(\mathcal{B}(M_i)))<\infty$ for all $i\in [l_c]$ if and only if $\gd(\End_{U_i}(\mathcal{B}(M_i)))\le 2$ for all $i\in[l_c]$ if and only if $M_i$ is an additive generator for $U_i$-mod for all $i\in [l_c]$ if and only if $|P_c(f(x))|=\max\{j\in P_c(f(x))\}$ for all $f(x)\in \mathcal{M}_c$ if and only if $\mathcal{I}_c=\emptyset$. Thus (ii) $\Leftrightarrow$ (v) $\Leftrightarrow$ (iv) $\Leftrightarrow $ (iii). Lemma \ref{cmoe} shows that $|P_c(f(x))|=\max\{j\in P_c(f(x))\}$ for all $f(x)\in \mathcal{M}_c$ if and only if $M_i$ is an additive generator for $U_i$-mod for all $i\in[l_c]$ if and only if $\det(C_{\End_{U_i}(\mathcal{B}(M_i))})=1$ for all $i\in[l_c]$. Thus (v) $\Leftrightarrow$ (vi). It follows from \cite[Statement 9]{dr} that (i) $\Ra$ (ii), and from \cite[Theorem 2]{dr} that (iii) $\Ra$ (i).  $\square$

\medskip
We have the following theorem.

\begin{Theo}\label{cdaqha}
Let $R$ be a field. Then the following hold.

$(1)$ The Cartan determinant conjecture holds for centralizer matrix algebras.

$(2)$ Let $c\in M_n(R)$ and $d\in M_m(R)$. If $S_n(c,R)$ and $S_m(d,R)$ are singularly equivalent, then

$\quad (i)$ $S_n(c,R)$ is quasi-hereditary if and only if $S_m(d,R)$ is quasi-hereditary.

$\quad (ii)$ The Cartan determinants of $S_n(c,R)$ and $S_m(d,R)$ are equal.
\end{Theo}

{\it Proof.} (1) follows from Lemma \ref{efqha}(2).

We prove (2). Suppose $S_n(c,R)$ and $S_m(d,R)$ are singularly equivalent. It follows from Theorem \ref{seocma} that $c$ and $d$ are $Sg$-equivalent. Hence $\mathcal{U}_c=\mathcal{U}_d$. Moreover, $\mathcal{I}_c=\emptyset$ if and only if $\mathcal{I}_d=\emptyset$. By Lemma \ref{efqha}(2), $S_n(c,R)$ is quasi-hereditary if and only if so is $S_m(d,R)$. This shows (2)(i). It remains to prove (2)(ii).

Suppose $\mathcal{I}_c=\emptyset$. Then $\mathcal{I}_d=\emptyset$ and it follows from Lemma \ref{efqha}(1) that $\det(C_{S_n(c,R)})=1=\det(C_{S_m(d,R)}).$

Suppose $\mathcal{I}_c\ne\emptyset$. Then $\mathcal{I}_d\ne\emptyset$. It follows from Lemma \ref{efqha}(1) and $\mathcal{U}_c=\mathcal{U}_d$ that
$$\det(C_{S_n(c,R)})
=\prod_{f(x)\in \mathcal{I}_c}\prod_{i\in \mathcal{D}_{P_c(f(x))}}i
=\prod_{i\in \mathcal{U}_{c}}i
=\prod_{i\in \mathcal{U}_{d}}i
=\prod_{g(x)\in \mathcal{I}_d}\prod_{i\in \mathcal{D}_{P_d(g(x))}}i
=\det(C_{S_m(d,R)}). \quad \square$$

Finally, we propose the following question.
\begin{Ques} How to classify all (basic) tilting (respectively, silting) modules over $S_n(c,R)$ for $R$ a field?
\end{Ques}

\medskip
\textbf{Acknowledgements.} The research work of the corresponding author C. C. Xi was supported partially by the National Natural Science Foundation of China (Grants 12031014).
Both authors thank Prof. Dr. Martin Kalck for pointing out the references \cite{kiwy1, kk1} to their attention when he read the first version of the preprint arXiv:2603.20643v1.

%
%
%
%
%
%
\end{document}